\documentclass[a4paper]{amsart}
\usepackage[ansinew]{inputenc}
\usepackage[T1]{fontenc}
\usepackage{longtable}
\usepackage{xcolor}
\usepackage{mathrsfs}
\usepackage{amssymb,amsmath,amsthm}
\usepackage[matrix, arrow, curve]{xy}
\usepackage{bbm}
\usepackage{tikz-cd}
\usepackage{stmaryrd}
\usepackage{mathtools}
\usetikzlibrary{quotes,babel, angles}

\newcommand{\prk}{\textnormal{prk}\,}

\newcommand{\morA}{\textnormal{mor}\,A}
\newcommand{\radA}{\textnormal{rad}_A}
\newcommand{\Hom}{\textnormal{Hom}_A}

\newcommand{\modA}{\textnormal{mod}\,A}
\renewcommand{\mod}{\textnormal{mod}\,}
\newcommand{\ModA}{\textnormal{Mod}\,A}
\newcommand{\im}{\textnormal{Im}\,}

\newcommand{\add}{\textnormal{add}\,}
\newcommand{\Ab}{\textnormal{Ab}}
\newcommand{\KG}{\textnormal{KG}}
\newcommand{\TD}{\textnormal{TD}}

\newcommand{\ob}{\textnormal{ob}\,}
\newcommand{\op}{\textnormal{op}}
\newcommand{\gen}{\textnormal{gen}\,}
\newcommand{\cogen}{\textnormal{cogen}\,}
\newcommand{\fun}[1]{(\mod #1, \Ab)}
\newcommand{\funfp}[1]{(\mod #1, \Ab)^\textnormal{fp}}
\begin{document}

\numberwithin{equation}{section} \theoremstyle{plain}
\newtheorem*{thm*}{Main Theorem}
\newtheorem{thm}{Theorem}
\numberwithin{thm}{section}
\newtheorem{coro}[thm]{Corollary}
\newtheorem*{coro*}{Corollary}
\newtheorem{conj}[thm]{Conjecture}
\newtheorem*{conj*}{Conjecture}
\newtheorem{lem}[thm]{Lemma}
\newtheorem*{lem*}{Lemma}
\newtheorem{prop}[thm]{Proposition}
\newtheorem*{prop*}{Proposition}
\newtheorem{rem}[thm]{Remark}
\newtheorem*{rem*}{Remark}
\newtheorem{exa}[thm]{Example}
\newtheorem*{exa*}{Example}
\newtheorem{df}[thm]{Definition}
\newtheorem*{df*}{Definition}
\newtheorem{ques}{Question}
\newtheorem*{ques*}{Question}
\newtheorem{construction}{Construction}
\newtheorem*{construction*}{Construction}
\newtheorem*{ack*}{ACKNOWLEDGEMENTS}
\newtheorem{case}{Case}

\title{Ideal Torsion Pairs for Artin Algebras} 
\author{Kevin Schlegel}
\date{}

\address{Kevin Schlegel, University of Stuttgart, Institute for Algebra and Number Theory, Pfaffenwaldring 57, 70569 Stuttgart, Germany}
\email{kevin.schlegel@iaz.uni-stuttgart.de}


\begin{abstract} For the module category of an Artin algebra, we generalize the notion of torsion pairs to ideal torsion pairs. Instead of full subcategories of modules, ideals of morphisms of the ambient category are considered. We characterize the functorially finite ideal torsion pairs, which are those fulfilling some nice approximation conditions, first through corresponding functors and then through the notion of ideals determined by objects introduced in this work. As an application of this theory, we generalize preprojective modules, introduce a new homological dimension, the torsion dimension, and establish its connection with the Krull-Gabriel dimension. In particular, it is shown that both dimensions coincide for hereditary Artin algebras.     
\end{abstract}

\maketitle

\section*{Introduction}
The concept of torsion pairs is a useful tool to study the representation theory of finite dimensional algebras (or more generel Artin algebras). Let $A$ be an Artin algebra and $\mod A$ the category of finitely generated (left) $A$-modules. Recall that a torsion pair $(\mathcal{T}, \mathcal{F})$ in $\mod A$ is a pair of full subcategories of $\modA$ such that $\Hom(\mathcal{T}, \mathcal{F}) = 0$ and every $M\in \modA$ admits a short exact sequence
\begin{align*}
    0 \longrightarrow L \longrightarrow M \longrightarrow N \longrightarrow 0
\end{align*}
with $L \in \mathcal{T}$ and $N \in \mathcal{F}$. Of particular interest are the functorially finite torsion pairs, which are those fulfilling some nice approximation conditions. They arise in the context of tilting theory \cite{Brenner} and $\tau$-tilting theory \cite{Iyama}. The aim of this work is threefold: 
\begin{itemize}
    \item We generalize the notion of torsion pairs. Instead of full subcategories of $\modA$, ideals of morphisms in $\modA$ will be considered. This leads to the concept of ideal torsion pairs (Section 2).
    \item We focus on functorially finite ideal torsion pairs, which are those fulfilling some nice approximation conditions similar to functorially finite torsion pairs. We classify them through corresponding subfunctors of the forgetful functor $\modA \rightarrow \Ab$ (Section 2) and the notion of ideals determined by objects introduced in this work (Section 3), following Auslander's work on morphisms determined by objects \cite{Auslander1}.
    \item We apply the theory developed, generalizing the notion of preprojective modules (Section 5) and establishing a connection between ideal torsion pairs and the Krull-Gabriel dimension (Section 6).\\ 
\end{itemize}

A pair $(\mathcal{I}, \mathcal{J})$ of ideals of morphisms in $\modA$ is an \emph{ideal torsion pair} if $\psi \varphi = 0$ for all $\varphi\in \mathcal{I}$, $\psi \in \mathcal{J}$ and every $M\in \modA$ admits a short exact sequence
\begin{align*}
    0 \longrightarrow L \xlongrightarrow{\varphi} M \xlongrightarrow{\psi} N \longrightarrow 0 \tag{$\ast$}
\end{align*}
with $\varphi \in \mathcal{I}, \psi \in \mathcal{J}$. Further, $\mathcal{I}$ is called a \emph{torsion ideal} and $\mathcal{J}$ a \emph{torsion-free ideal}. The notion of ideal torsion pairs was parallely developed by the first three authors in \cite{Zhu}. To see that ideal torsion pairs generalize torsion pairs, let $(\mathcal{T}, \mathcal{F})$ be a torsion pair in $\modA$. Then $(\langle \mathcal{T} \rangle, \langle \mathcal{F} \rangle)$ is an ideal torsion pair, where $\langle \mathcal{T} \rangle$ and $\langle \mathcal{F} \rangle$ denote the collection of all morphisms factoring through a module in $\mathcal{T}$ and $\mathcal{F}$ respectively (Remark \ref{torsion}). This assignment from torsion pairs to ideal torsion pairs is injective.

The existence of the short exact sequences $(\ast)$ for ideal torsion pairs is functorial in $M$, which leads to a one to one correspondece between ideal torsion pairs $(\mathcal{I}, \mathcal{J})$ and subfunctors $t$ of the forgetful functor $\modA \rightarrow \Ab$ such that $tM$ is a submodule of $M$ for all $M\in \modA$ (Proposition \ref{subfun}). This correspondence is a crucial tool for the analysis of ideal torsion pairs, as we can now consider finitely presented functors to study them. The following theorem is a generalization of a result on torsion pairs \cite{Smalo}.
\\
\\
\textbf{Theorem A} (Theorem \ref{ffiff}). \textit{Let $(\mathcal{I}, \mathcal{J})$ be an ideal torsion pair and $t$ the corresponding functor. Then $\mathcal{I}$ is functorially finite if and only if $t$ is finitely presented if and only if $\mathcal{J}$ is functorially finite.}
\\

The definition of functorially finite ideals is given in Section 1. An ideal torsion pair is \emph{functorially finite} if it satisfies the equivalent properties in Theorem A. Inspired by Auslander's notion of morphisms determined by objects \cite{Auslander1}, an ideal $\mathcal{I}$ is \emph{right $C$-determined} for $C\in \modA$ if $\varphi f \in \mathcal{I}$ for all $f$ starting in $C$ (such that the composition is defined) already implies $\varphi \in \mathcal{I}$. Dually $\mathcal{I}$ is \emph{left $C$-determined} if $f \varphi \in \mathcal{I}$ for all $f$ ending in $C$ already implies $\varphi \in \mathcal{I}$. As it turns out, the right \mbox{$A$-determined} ideals of $\modA$ are precisely the torsion ideals (Proposition \ref{adet}). One might ask, when a torsion ideal is also left $C$-determined for some $C\in \modA$. The following result answers this question and gives a second viewpoint on functorially finite ideal torsion pairs.
\\
\\
\textbf{Theorem B} (Theorem \ref{det}).
\begin{itemize}
    \item[\rm (a)]\textit{In $\modA$ we have an equality 
    \begin{align*}
        \left\{\begin{matrix}\text{functorially finite}\\ \text{torsion ideals}\end{matrix}\right\} = \bigcup_{C \in \modA } \left\{\begin{matrix}\text{left $C$-determined} \\ \text{torsion ideals}\end{matrix}\right\}.
    \end{align*}}
    \item[\rm (b)] \textit{For $C\in \modA$ there exists a one to one correspondence
    \begin{align*}
        \left\{\begin{matrix}\text{left $C$-determined} \\ \text{torsion ideals}\end{matrix}\right\} \longleftrightarrow \left\{\begin{matrix}
 \text{bi-submodules of} \\ 
 \text{ ${}_AC_{\textnormal{End}_A(C)^{\textnormal{op}}}$ }
\end{matrix} \right\}.
    \end{align*}}\\
\end{itemize}

The main application of ideal torsion pairs in this work is their relation with the Krull-Gabriel dimension of $A$ (see \cite{Geigle}), denoted by $\KG(A)$. In a sense, $\KG(A)$ measures a complexity related to the representation type of $\modA$: We have $\KG (A) = 0$ if and only if $A$ is of finite representation type by a classical result of Auslander \cite{Auslander}, and $\KG (A) \neq 1$ by a result of Herzog \cite{Herzog} and Krause \cite[Corollary 11.4]{Krause2}. Geigle proved for hereditary $A$ that $\KG(A) = 2$ if $A$ is tame and $\KG(A) = \infty$ if $A$ is wild \cite{Geigle}. If $A$ is a finite dimensional $k$-algebra over a field $k$, then $A$ is conjectured to be tame domestic if and only if $\KG(A) < \infty$ \cite[Conjecture 3]{Schroer}. For example, this conjecture is proven for string algebras over an algebraically closed field \cite[Corollary 1.2]{Laking}.

We introduce a new dimension, the \emph{torsion dimension} of $A$, denoted by $\TD(A)$. It is defined to be the m-dimension (see \cite{Prest88}) of the lattice of functorially finite ideal torsion pairs $(\mathcal{I}, \mathcal{J})$, where $(\mathcal{I}, \mathcal{J}) \leq (\mathcal{I}', \mathcal{J}')$ if $\mathcal{I} \subseteq \mathcal{I}'$. We always have $\TD(A) \leq \KG(A)$ and equality if $A$ is commutative (Proposition \ref{small} and Remark \ref{commutative}). A central link between the torsion dimension and the Krull-Gabriel dimension of $A$ is given by the radical ideal $\radA$, which is the smallest ideal containing all non-isomorphisms between indecomposable modules in $\modA$, and its powers $\radA^\alpha$ for ordinal numbers $\alpha$ (see \cite{Prest}). The following conjecture is an "ordinal version" of Schr\"oers conjecture \cite[Conjecture 5]{Schroer}.
\\
\\
\textbf{Conjecture C} (Conjecture \ref{schr}). \textit{Let $\alpha$ be a non-zero ordinal number. Then $\KG(A) = \alpha+1$ if and only if $\radA^{\omega \alpha} \neq 0$ and $\radA^{\omega ({\alpha+1})} = 0$, where $\omega$ denotes the first non-finite ordinal.}\\

One step towards this conjecture is a result by Krause \cite[Corollary 8.14]{Krause}: If $\radA^{\omega \alpha} \neq 0$, then $\KG(A) \geq \alpha$. We show a similar result for the torsion dimension: If $\radA^{\omega \alpha} \neq 0$, then $\TD(A) \geq \alpha$ (Proposition \ref{okay}). Under the assumption of Conjecture C it then follows that $\KG(A) = \TD(A)$ or $\KG(A) = \TD(A)+1$ (Corollary \ref{kgtd}). There are no known examples, where the second equality holds. Conjecture C is proven if $A$ is a string algebra over an algebraically closed field \cite[Corollary 1.3]{Laking}. Lastly, if there exists $M\in \modA$ such that the smallest torsion class containing $M$ is not functorially finite, then $\TD(A) > 1$ (Proposition \ref{notone}). Putting everything together, we can calculate the values $\TD(A)$ for all hereditary Artin algebras and verify $\TD(A) = \KG(A)$ in this case.
\\
\\
\textbf{Theorem D} (Theorem \ref{here}). \textit{Let $A$ be a hereditary Artin algebra.} 
\begin{itemize}
    \item[(a)] \textit{If $A$ is representation finite, then $\TD(A) = 0$.}
    \item[(b)] \textit{If $A$ is tame, then $\TD(A) = 2$.}
    \item[(c)] \textit{If $A$ is wild, then $\TD(A) = \infty$.}\\
\end{itemize}

A second application of ideal torsion pairs is a generalization of preprojective modules \cite{AuslanderSmalo}. Again, the radical ideal and its ordinal powers appear in this context. For an ordinal number $\alpha$ let $\rm{I}(\radA^\alpha)$ be the smallest torsion ideal containing $\radA^\alpha$. The \emph{projective rank} of $M\in \modA$ is the smallest $\alpha$ such that the identity $1_M$ is contained in $\rm{I}(\radA^\alpha)$ (or $\infty$ if no such $\alpha$ exists). The modules of projective rank $0$ are precisely the projective modules and those of finite projective rank are precisely the preprojective modules (Proposition \ref{preproj}). If $\lambda\neq 0$ is a limit ordinal (or $\infty$), then there exists a module of projective rank greater than or equal to $\lambda$ if and only if $\radA^\lambda \neq 0$ (Corollary \ref{exi}).  As it turns out, the behaviour of modules of non-finite projective rank is opposite to the finite case: 
\\
\\
\textbf{Corollary E} (Corollary \ref{unbounded}). \emph{Let $\lambda$ be a non-zero limit ordinal and $n\in \mathbbm{N}$. If there exists a module of projective rank between $\lambda$ and $\lambda+n$, then there exist infinitely many indecomposable modules of projective rank between $\lambda$ and $\lambda+n$ and their length is unbounded.}\\

\textbf{Acknowledgements.} I would like to express my gratitude to Frederik Marks for weekly discussions of the mathematical content in this work. I am also thankful to him for guidance while writing this article.





\section{Preliminaries}

Let $A$ be an Artin $k$-algebra ($k$ is a commutative artinian ring and $A$ is finitely generated over $k$), $\modA$ the category of finitely generated (left) $A$-modules and $\Ab$ the category of abelian groups. Further, let $\mod k$ be the category of finitely generated $k$-modules and $D = \text{Hom}_k (-,I)$, where $I$ is the injective hull of $k/ J_k$ with $J_k$ the Jacobson radical of $k$. Then $D$ induces a duality between $\mod k$ and $\mod k$ as well as between $\mod A$ and $\mod A^\textnormal{op}$.

The main objects of concern in this work are the (functorially finite) ideal torsion pairs, a generalization of (functorially finite) torsion pairs. In what follows, we present the classical setup.

\subsection*{Torsion pairs}

A pair $(\mathcal{T}, \mathcal{F})$ of full subcategories of $\modA$ is a \emph{torsion pair} if
\begin{itemize}
    \item[(i)] $\Hom(M,N) = 0$ for all $M\in \mathcal{T}$ and $N\in \mathcal{F}$, and
    \item[(ii)] every $M\in \modA$ admits a short exact sequence
    \begin{align*}
        0 \longrightarrow L \longrightarrow M \longrightarrow N \longrightarrow 0
    \end{align*}
    with $L\in \mathcal{T}$ and $N\in \mathcal{F}$.
\end{itemize}
Further, $\mathcal{T}$ is called a \emph{torsion class} and $\mathcal{F}$ a \emph{torsion-free class}. Instead of (ii), one can define torsion pairs by demanding a maximality condition on $\mathcal{T}$ and $\mathcal{F}$ with respect to the orthogonality property (i), that is if $M\notin \mathcal{T}$, then there is $N\in \mathcal{F}$ with $\Hom(M,N) \neq 0$ and if $N\notin \mathcal{F}$, then there is $M\in \mathcal{T}$ with $\Hom(M,N) \neq 0$. This turns out to be an equivalent definition of torsion pairs. The following result is a well-known characterization of torsion classes and torsion-free classes in $\modA$.

\begin{lem}\label{tclosure} A full subcategory $\mathcal{C}$ of $\modA$ is a torsion(-free) class if and only if $\mathcal{C}$ is closed under extensions and factor modules (submodules).
\end{lem}

\subsection*{Subcategories and approximations}
We discuss approximations of modules with respect to full additive subcategories of $\modA$ as in \cite{AuslanderSmalo}. A morphism $f$ in $\modA$ is \emph{left (right) minimal} if $\alpha f = f$ (respectively $f\alpha = f$) implies that $\alpha$ is an isomorphism for all $\alpha$ such that the composition is defined. Let $\mathcal{C}$ be a full additive subcategory of $\modA$. A morphism $\varphi\colon M \rightarrow C_M$ is a \textit{left $\mathcal{C}$-approximation} of $M$ if $C_M \in \mathcal{C}$ and every morphism $M\rightarrow C$ with ${C} \in \mathcal{C}$ factors through $\varphi$, that is there exists a morphism $C_M\rightarrow C$ such that the diragram
\begin{equation*}
    \begin{tikzcd}
        M \arrow[r, "\varphi"] \arrow[rd] & C_M \arrow[d]\\
        & C 
    \end{tikzcd}
\end{equation*}
commutes. If $\varphi$ is also left minimal, then $\varphi$ is a \emph{left minimal $\mathcal{C}$-approximation}. If every $M\in \modA$ admits a left $\mathcal{C}$-approximation, then $\mathcal{C}$ is \textit{covariantly finite}. Dually, a morphism $\psi\colon C_N \rightarrow N$ is a \textit{right $\mathcal{C}$-approximation} of $N$ if $C_N\in \mathcal{C}$ and every morphism $C\rightarrow N$ with ${C} \in \mathcal{C}$ factors through $\psi$, that is there exists a morphism $C\rightarrow C_N$ such that the diragram
\begin{equation*}
    \begin{tikzcd}
        C_N \arrow[r, "\psi"] & N\\
        C \arrow[ru] \arrow[u] &
    \end{tikzcd}
\end{equation*}
commutes. If $\psi$ is also right minimal, then $\psi$ is a \emph{right minimal $\mathcal{C}$-approximation}. If every $N\in \modA$ admits a right $\mathcal{C}$-approximation, then $\mathcal{C}$ is \textit{contravariantly finite}. If $\mathcal{C}$ is both co- and contravariantly finite, then $\mathcal{C}$ is \textit{functorially finite}. Further, it is well-known that if $M$ admits a left (right) $\mathcal{C}$-approximation, then $M$ also admits a left (right) minimal $\mathcal{C}$-approximation.

Let $(\mathcal{T}, \mathcal{F})$ be a torsion pair and 
\begin{align*}
    0 \longrightarrow L \longrightarrow{} M \longrightarrow{} N \longrightarrow{} 0 
\end{align*}
a short exact sequence with $M\in \modA$ arbitrary, $L \in \mathcal{T}$ and $N \in \mathcal{F}$. Then the monomorphism $L \rightarrow M$ is always a right minimal $\mathcal{T}$-approximation and the epimorphism $M \rightarrow N$ a left minimal $\mathcal{F}$-approximation. In particular $\mathcal{T}$ is contravariantly finite and $\mathcal{F}$ covariantly finite. For $M\in \modA$ we denote by $\gen M$ ($\cogen M$) the collection of all modules that are isomorphic to a factor module (submodule) of $M^n$ for some $n\in \mathbbm{N}$. One might ask, when $\mathcal{T}$ and $\mathcal{F}$ are functorially finite. This is answered by the following result.

\begin{thm} {\rm \cite{Smalo}} Let $(\mathcal{T}, \mathcal{F})$ be a torsion pair in $\modA$. The following are equivalent. 
\begin{itemize}
    \item[\rm (i)] The torsion class $\mathcal{T}$ is functorially finite. 
    \item[\rm (ii)] There exists $M \in \modA$ such that $\mathcal{T} = \gen M$.
    \item[\rm (iii)] The torsion-free class $\mathcal{F}$ is functorially finite.
    \item[\rm (iv)] There exists $N \in \modA$ such that $\mathcal{F} = \cogen N$. 
\end{itemize}
\end{thm}

A torsion pair fulfilling the equivalent properties in the above theorem is a \emph{functorially finite} torsion pair.


\subsection*{Ideals and approximations}

Passing from torsion pairs to ideal torsion pairs, we switch from full additive subcategories of $\modA$ to ideals of morphisms in $\modA$. In the later context, approximations with respect to ideals of morphisms in $\modA$ are relevant. For this we follow \cite{Fu} but switch the notation. Recall that a class of morphisms $\mathcal{I}$ in $\modA$ is an \textit{ideal} of $\modA$ if
\begin{itemize}
    \item[(i)] for all $\varphi, \psi \in \mathcal{I}$ we have $\varphi + \psi \in \mathcal{I}$ (if the addition is defined), and
    \item[(ii)] for all $\varphi\in \mathcal{I}$ and all $f,g$ we have $g \varphi f \in \mathcal{I}$ (if the composition is defined). 
\end{itemize}
For $M,N\in \modA$ we denote by $\mathcal{I}(M,N)$ the collection of all morphisms in $\mathcal{I}$ starting in $M$ and ending in $N$. Now $\mathcal{I}$ induces the additive functors $\mathcal{I}(M,-) \colon \modA \rightarrow \Ab$ and $\mathcal{I}(-,N) \colon \modA \rightarrow \Ab^\op$ in the canonical way.

A morphism $\varphi\colon M \rightarrow C_M$ is a \textit{left $\mathcal{I}$-approximation} of $M$ if $\varphi \in \mathcal{I}$ and every morphism $M\rightarrow C$ in $\mathcal{I}$ factors through $\varphi$, that is there exists a morphism $C_M\rightarrow C$ such that the diragram
\begin{equation*}
    \begin{tikzcd}
        M \arrow[r, "\varphi"] \arrow[rd] & C_M \arrow[d]\\
        & C 
    \end{tikzcd}
\end{equation*}
commutes. If $\varphi$ is also left minimal, then $\varphi$ is a \emph{left minimal $\mathcal{I}$-approximation}. If every $M\in \modA$ admits a left $\mathcal{I}$-approximation, then $\mathcal{I}$ is \textit{covariantly finite}. Dually, a morphism $\psi\colon C_N \rightarrow N$ is a \textit{right $\mathcal{I}$-approximation} of $N$ if $\psi \in \mathcal{I}$ and every morphism $C\rightarrow N$ in $\mathcal{I}$ factors through $\psi$, that is there exists a morphism $C\rightarrow C_N$ such that the diragram
\begin{equation*}
    \begin{tikzcd}
        C_N \arrow[r, "\psi"] & N\\
        C \arrow[ru] \arrow[u] &
    \end{tikzcd}
\end{equation*}
commutes. If $\psi$ is also right minimal, then $\psi$ is a \emph{right minimal $\mathcal{I}$-approximation}. If every $N\in \modA$ admits a right $\mathcal{I}$-approximation, then $\mathcal{I}$ is \textit{contravariantly finite}. If $\mathcal{I}$ is both co- and contravariantly finite, then $\mathcal{I}$ is \textit{functorially finite}. Further, if $M$ admits a left (right) $\mathcal{C}$-approximation, then $M$ also admits a left (right) minimal $\mathcal{C}$-approximation.

Given a full additive subcategory $\mathcal{C}$ of $\modA$, we can associate an ideal $\langle \mathcal{C} \rangle $ of all morphisms factoring through modules in $\mathcal{C}$. In this way, the notion of approximations with respect to ideals of $\modA$ generalizes the notion of approximations with respect to full additive subcategories of $\modA$ by the following result.

\begin{lem}\label{iff} Let $\mathcal{C}$ be a full additive subcategory of $\modA$. 
\begin{itemize}
    \item[\rm (a)] Every left (right) $\mathcal{C}$-approximation is also a left (right) $\langle \mathcal{C} \rangle$-approximation.
    \item[\rm (b)] Every left (right) minimal $\langle \mathcal{C} \rangle$-approximation is also a left (right) minimal $\mathcal{C}$-approximation.
\end{itemize}
\end{lem}

\begin{proof} We only show the "left case". Let $\varphi \colon M \rightarrow C_M$ be a left $\mathcal{C}$-approximation. Then $\varphi$ factors through $C_M \in \mathcal{C}$, so $\varphi \in \langle \mathcal{C} \rangle$. Further, if $\psi \colon M \rightarrow C$ is in $\langle \mathcal{C} \rangle$ then it factors through a module in $\mathcal{C}$. Hence, $\psi$ must factor through the left $\mathcal{C}$-approximation $\varphi$. It follows that $\varphi$ is a left $\langle \mathcal{C} \rangle$-approximation.

Let $\varphi\colon M \rightarrow C_M$ be a left minimal $\langle \mathcal{C} \rangle$-approximation. Given $\psi \colon M \rightarrow C$ with $C\in \mathcal{C}$, clearly $\psi \in \langle \mathcal{C} \rangle$ so $\psi$ factors through $\varphi$. It is left to show $C_M\in \mathcal{C}$ for $\varphi$ to be a left minimal $\mathcal{C}$-approximation. Since $\varphi \in \langle \mathcal{C} \rangle$, there exists $f \colon M \rightarrow C$ and $g \colon C \rightarrow C_M$ with $C\in \mathcal{C}$ and $\varphi = gf$. Now $f$ must factor through $\varphi$, that is there exists $h\colon C_M\rightarrow C$ with $f = h \varphi$. Thus,   $\varphi = gh\varphi$. Because $\varphi$ is left minimal, it follows that $gh$ is an isomorphism. Hence, $C_M$ is a direct summand of $C$ and $C_M \in \mathcal{C}$. It follows that $\varphi$ is a left minimal $\mathcal{C}$-approximation. 
\end{proof}

\subsection*{The radical ideal}

A morphism $\varphi\colon M \rightarrow N$ in $\modA$ is \emph{radical} if for all indecomposable modules $X\in \modA$ and morphisms $f\colon X \rightarrow M, g\colon N \rightarrow X$ the composition $g \varphi f$ is a non-isomorphism. The \emph{radical ideal} of $\modA$ is the ideal consisting of all radical morphisms in $\modA$, denoted by $\radA$. It plays an important role in the application of ideal torsion pairs to the Krull-Gabriel dimension and the generalization of preprojective modules (Section 5 and 6).

A morphism $\varphi \colon X \rightarrow M$ in $\modA$ is \emph{left almost split} if $\varphi$ is not a split monomorphism and every morphism $X \rightarrow M'$ that is not a split monomorphism factors through $\varphi$. In that case $X$ must be indecomposable and morphisms starting in $X$ that are not split monomorphisms are precisely the radical morphisms. It follows that left almost split morphisms are left $\radA$-approximations starting in indecomposable modules. Arbitrary left $\radA$-approximations can be constructed out of those. Hence, the existence of left almost split morphisms (see \cite{AuslanderReiten}) implies that $\radA$ is covariantly finite. Dually, the existence of right almost split morphisms implies that $\radA$ is contraviantly finite.  

\begin{prop}\label{radical} The ideal $\radA$ is functorially finite. 
\end{prop}

\subsection*{Morphisms determined by objects}
In \cite{Auslander1} and \cite{Auslander2}, Auslander introduced the concept of morphisms determined by objects. In Section 3 we will introduce the notion of ideals determined by objects to give a second viewpoint on functorially finite ideal torsion pairs. Further, we establish a connection with morphisms determined by objects.  

A morphism $\varphi\colon M \rightarrow N$ in $\modA$ is \emph{left $C$-determined} for $C\in \modA$ provided the following condition is satisfied: For every morphism $\varphi' \colon M \rightarrow N'$ in $\modA$, if $f \varphi'$ factors through $\varphi$ for all $f \colon N' \rightarrow C$, then $\varphi'$ already factors through $\varphi$. As an example, if $\varphi$ is left almost split, then $\varphi$ is left $M$-determined. Dually, a morphism $\varphi\colon M \rightarrow N$ in $\modA$ is \emph{right $C$-determined} provided the following condition is satisfied: For every morphism $\varphi' \colon M' \rightarrow N$ in $\modA$, if $\varphi' f$ factors through $\varphi$ for all $f \colon C \rightarrow M'$, then $\varphi'$ already factors through $\varphi$. The following result is due to Auslander, with a slight correction by Ringel \cite{Ringel}.

\begin{thm}\label{ar} Let $\varphi$ be a morphism in $\modA$, $K$ its kernel and $Q$ its cokernel.
\begin{itemize}
    \item[\rm (a)] Then $\varphi$ is left $\tau Q \oplus I$-determined, where $\tau$ denotes the Auslander-Reiten translation and $I$ the injective hull of the top of $K$.
    \item[\rm (b)] Then $\varphi$ is right $\tau^{-} K \oplus P$-determined, where $\tau^{-}$ denotes the inverse of the Auslander-Reiten translation and $P$ the projective cover of the socle of $Q$.
\end{itemize}
\end{thm}

\subsection*{The functor category}
Let $\fun{A}$ be the abelian category of additive functors $\modA \rightarrow \Ab$ and $\funfp{A}$ the full subcategory of finitely presented functors in $\fun{A}$.  Recall that $F\in \fun{A}$ is \emph{finitely presented} if there exists a short exact sequence 
\begin{align*}
    \Hom(N,-) \longrightarrow \Hom(M,-) \longrightarrow F \longrightarrow 0
\end{align*}
with $M,N \in \modA$. A functor $F \in \fun{A}$ is \emph{finitely generated} if it is a factor of $\Hom(M,-)$ for some $M\in \modA$. The functor categories are well-studied (see for example \cite{Prest2}). The category $\funfp{A}$ is closed under extensions, kernels and cokernels in $\fun{A}$. As a consequence, it inherits the abelian structure of $\fun{A}$. The functor categories will help us to study (functorially finite) ideal torsion pairs (Section 2). Further, the definition of the Krull-Gabriel dimension takes place in $\funfp{A}$.

For $F\in \fun{A}$ and $M\in \modA$, since mulitplication with an element in $k$ is a morphism in $\modA$, we can consider $F(M)$ as a $k$-module. If $F$ is finitely presented, then $F(M) \in \mod k$. Now the duality $D$ induces a duality between $\funfp{A}$ and $\funfp{A^\op}$:  

\begin{lem}\label{duality}{\rm \cite[Proposition 3.3]{Auslander3}} There is a duality $d$ between $\funfp{A}$ and $\funfp{A^\op}$ given by $F \mapsto dF$ with $dF(M) = DF(DM)$.
\end{lem}

\subsection*{The Krull-Gabriel dimension}

Let $\mathcal{A}$ be an abelian category. A \emph{Serre subcategory} $\mathcal{S}$ of $\mathcal{A}$ is a full subcategory closed under extensions, subobjects and factor objects. The quotient category $\mathcal{A}/\mathcal{S}$ is again an abelian category (for more details, see \cite{Gabriel}).

Following Geigle \cite{Geigle}, the \emph{Krull-Gabriel dimension}, $\KG(A)$, of $A$ is defined as follows: Let $\mathcal{S}_{-1} = 0$ be the trivial Serre subcategory of $\funfp{A}$. If $\alpha$ is an ordinal of the form $\alpha = \beta +1$, let $\mathcal{S}_\alpha$ be the Serre subcategory of all objects in $\funfp{A}$ which become of finite length in $\funfp{A} /\mathcal{S}_{\beta}$. If $\lambda$ is a limit ordinal, then let $\mathcal{S}_{\lambda} = \bigcup_{\alpha < \lambda} \mathcal{S}_{\alpha}$. Now the Krull-Gabriel dimension of $A$ equals the smallest ordinal $\alpha$ with $\mathcal{S}_{\alpha} = \funfp{A}$. If no such $\alpha$ exists, then $\KG(A) = \infty$.

In Section 6 we connect the theory of ideal torsion pairs with the Krull-Gabriel dimension. The Krull-Gabriel dimension is an important homological dimension, as its value is connected with the representation type of $\modA$. We have $\KG (A) = 0$ if and only if $A$ is of finite representation type by a classical result of Auslander \cite{Auslander}. If $A$ is hereditary, then Geigle showed:

\begin{thm}{\rm \cite{Geigle}} Let $A$ be a hereditary Artin algebra. If $A$ is tame, then $\KG(A) = 2$ and if $A$ is wild, then $\KG(A) = \infty$.
\end{thm}

\subsection*{The m-dimension of a modular lattice}
Let $(L, \lor, \land)$ be a lattice that is modular, which means $a \lor (x \land b) = (a \lor x) \land b$ for all $a,b,x \in L$ with $a \leq b$. For example, the collection of all subobjects of a fixed object in an abelian category is a modular lattice. For $x,y\in L$ we define $x\sim y$ if the interval $[x\land y, x\lor y]$ has finite length. Then $\sim$ defines an equivalence relation on $L$ such that $L/{\sim}$ is again a modular lattice and the canonical map $L \rightarrow L/{\sim}$ is a lattice homomorphism. Following Prest \cite{Prest88}, the \emph{m-dimension}, $\dim L$, of $L$ is defined as follows: Let $L_{-1} = L$. If $\alpha$ is an ordinal number of the form $\alpha = \beta +1$, let $L_{\alpha} = L_{\beta}/{\sim}.$ If $\lambda$ is a limit ordinal, then let $L_{\lambda} = \varinjlim {}_{\alpha < \lambda} L_{\alpha}$. Now the m-dimension of $L$ equals the smallest ordinal $\alpha$ such that $L_\alpha$ consists of exactly one element. If no such $\alpha$ exists, then $\dim L = \infty$. The following result yields a different way to compute $\KG(A)$ that will be important for the connection between ideal torsion pairs and the Krull-Gabriel dimension.

\begin{prop}\label{mdim}{\rm \cite[Proposition 7.2]{Krause}} Let $L$ be the modular lattice of finitely presented subfunctors of $\Hom(A,-)$. Then $\KG(A) = \dim L$.
\end{prop}

\section{Ideal torsion pairs}

A pair of ideals $(\mathcal{I}, \mathcal{J})$ in $\modA$ is an \textit{ideal torsion pair} if 
\begin{itemize}
    \item[(i)] $\psi \varphi = 0$ for all $\varphi \in \mathcal{I}$ and $ \psi \in \mathcal{J}$ (if the composition is defined), and
    \item[(ii)] every $M\in \modA$ admits a short exact sequence
\begin{align*}
0 \longrightarrow L \xlongrightarrow{\varphi}  M \xlongrightarrow{\psi} N \longrightarrow 0
\end{align*}
with $\varphi \in \mathcal{I}$ and $\psi \in \mathcal{J}$.
\end{itemize}
Further, $\mathcal{I}$ is called a \textit{torsion ideal} and $\mathcal{J}$ a \textit{torsion-free ideal}. The concept of ideal torsion pairs was parallelly developed by the first three authors in \cite{Zhu} in a more general context.

\begin{rem}\label{itprem} \rm Let $(\mathcal{I}, \mathcal{J})$ be an ideal torsion pair, $M\in \modA$ and 
\begin{align*}
0 \longrightarrow L \xlongrightarrow{\varphi}  M \xlongrightarrow{\psi} N \longrightarrow 0
\end{align*}
a short exact sequence with $\varphi \in \mathcal{I}$ and $\psi \in \mathcal{J}$. Then for $f\colon X \rightarrow M$ in $\mathcal{I}$ we have $\psi f = 0$, so $f$ must factor through $\varphi$. Hence, $\varphi$ is a right $\mathcal{I}$-approximation. Similarly $\psi$ is a left $\mathcal{J}$-approximation. In particular $\mathcal{I}$ is contravarinatly finite and $\mathcal{J}$ covariantly finite.
\end{rem}

For an ideal $\mathcal{I}$ we denote by $\mathcal{I}^\perp$ the collection of all morphisms $\psi$ with $\psi \varphi = 0$ for all $\varphi\in \mathcal{I}$ and dually ${}^\perp \mathcal{I}$. One can easily verify that $\mathcal{I}^\perp$ and ${}^\perp \mathcal{I}$ are ideals. Similar to torsion pairs, also ideal torsion pairs can be defined by replacing (ii) with a maximality condition on $\mathcal{I}$ and $\mathcal{J}$ with respect to the orthogonality property (i). This is shown by the following result.

\begin{prop}\label{ortho}Let $\mathcal{I}$ and $\mathcal{J}$ be ideals in $\modA$.
\begin{itemize}
    \item[\rm (a)] The pair $({}^\perp (\mathcal{I}^\perp), \mathcal{I}^\perp)$ is an ideal torsion pair.
    \item[\rm (b)] The pair $({}^\perp \mathcal{J}, ({}^\perp \mathcal{J})^\perp)$ is an ideal torsion pair.
    \item[\rm (c)] The pair $(\mathcal{I}, \mathcal{J})$ is an ideal torsion pair if and only if $\mathcal{I}^\perp = \mathcal{J}$ and ${}^\perp \mathcal{J} = \mathcal{I}$.
\end{itemize}
\end{prop}

\begin{proof} (a) Clearly $\psi \varphi = 0$ for all $\varphi \in {}^\perp (\mathcal{I}^\perp)$ and $\psi \in \mathcal{I}^\perp$. For $M\in \modA $ let $L \subseteq M$ be the sum of all images of morphisms in $\mathcal{I}$ ending in $M$. Then $M \rightarrow M/L$ is contained in $\mathcal{I}^\perp$. It is left to show that the inclusion $\varphi \colon L \rightarrow M$ is contained in ${}^\perp (\mathcal{I}^\perp)$. Suppose there exists $\psi \in \mathcal{I}^\perp$ with $\psi \varphi \neq 0$. Then by the definition of $L$ there exists $\varphi'\colon X \rightarrow M$ in $\mathcal{I}$ with $\psi \varphi' \neq 0$, which is a contradiction. Thus,   always $\psi \varphi = 0$ and $\varphi \in {}^\perp (\mathcal{I}^\perp)$.

(b) Similar to (a).

(c) If $\mathcal{I}^\perp = \mathcal{J}$ and ${}^\perp \mathcal{J} = \mathcal{I}$, then $({}^\perp (\mathcal{I}^\perp), \mathcal{I}^\perp)$ equals $(\mathcal{I}, \mathcal{J})$ and is an ideal torsion pair by (a). If $(\mathcal{I}, \mathcal{J})$ is an ideal torsion pair, then $ \psi \varphi = 0$ for all $\varphi \in \mathcal{I}$ and $\psi \in \mathcal{J}$. Thus,   $\mathcal{I}^\perp \supseteq \mathcal{J}$. Now let $f \colon M \rightarrow N$ be a morphism in $\mathcal{I}^\perp$ and consider a short exact sequence
\begin{align*}
0 \longrightarrow L \xlongrightarrow{\varphi}  M \xlongrightarrow{\psi} N \longrightarrow 0
\end{align*}
with $\varphi \in \mathcal{I}$ and $\psi \in \mathcal{J}$. Then $f \varphi = 0$ implies that $f$ factors through $\psi$. Hence,  $f \in \mathcal{J}$. It follows that $\mathcal{I}^\perp = \mathcal{J}$ and similarly ${}^\perp \mathcal{J} = \mathcal{I}$.
\end{proof}

As for torsion classes and torsion-free classes, we have an interal characterization for torsion ideals and torsion-free ideals. However, in each case it only requires one closure property instead of two (compare Lemma \ref{tclosure}). This already hints at the fact that, in general, there are much more ideal torsion pairs then torsion pairs.

\begin{lem}\label{charai}\begin{itemize}\item[\rm (a)] An ideal $\mathcal{I}$ of $\modA$ is a torsion ideal if and only if $\varphi f \in \mathcal{I}$ implies $\varphi \in \mathcal{I}$ for all morphisms $\varphi$ and epimorphisms $f$.
    \item[\rm (b)] An ideal $\mathcal{J}$ of $\modA$ is a torsion-free ideal if and only if $g \psi \in \mathcal{J}$ implies $\psi \in \mathcal{J}$ for all morphisms $\psi$ and monomorphisms $g$.
\end{itemize}
\end{lem}

\begin{proof} (a) Let $(\mathcal{I}, \mathcal{J})$ be an ideal torsion pair and $\varphi f \in \mathcal{I}$ for a morphism $\varphi$ and an epimorphism $f$. Then $\psi \varphi f = 0$ for all $\psi \in \mathcal{J}$. Because $f$ is epic, also $\psi \varphi = 0$ for all $\psi \in \mathcal{J}$. Now Proposition \ref{ortho} (c) implies $\varphi \in \mathcal{I}$.

For the other implication, let $\mathcal{I}$ be an ideal such that $\varphi f\in\mathcal{I}$ implies $\varphi \in \mathcal{I}$ for all morphisms $\varphi$ and epimorphisms $f$. We show that $(\mathcal{I}, \mathcal{I}^\perp)$ is an ideal torsion pair. Similar to the proof of Proposition \ref{ortho}, for $M\in \modA$ let $L\subseteq M$ be the sum of all images of morphisms in $\mathcal{I}$ ending in $M$. Then $M \rightarrow M/L$ is contained in $\mathcal{I}^\perp$. It is left to show that the inclusion $\varphi \colon L \rightarrow M$ is contained in $\mathcal{I}$. Because $M$ is of finite length, it follows that $L$ is a finite sum and there exist $\varphi_i' \colon X_i \rightarrow M$ in $\mathcal{I}$ such that the image of the induced morphism $\varphi'\colon \bigoplus_{i=1}^n X_i \rightarrow M$ equals $L$. Now $\varphi_i' \in \mathcal{I}$ implies $\varphi' \in \mathcal{I}$ and if $f\colon \bigoplus_{i=1}^n X_i \rightarrow L$ denotes the projection onto the image of $\varphi'$, then $\varphi f = \varphi' \in \mathcal{I}$. Since $f$ is an epimorphism, the morphism $\varphi$ is contained in $\mathcal{I}$.

(b) Similar to (a).
 \end{proof}

Next, we want to show that for an ideal torsion pair $(\mathcal{I}, \mathcal{J})$ the ideal $\mathcal{I}$ is functorially finite if and only if so is the ideal $\mathcal{J}$. The same result holds for torsion pairs \cite{Smalo}, however its proof heavily relies on tilting theory which is not available in our context. Instead, we relate ideal torsion pairs to certain functors and apply theory of the functor category $\funfp{A}$ to deduce the desired result.

We denote by $\mathbbm{1}_A$ the identity functor $\modA \rightarrow \modA$ and view $\mathbbm{1}_A$ as a functor in the abelian category of additive functors from $\modA$ to $\modA$. Hence, a subfunctor $t$ of $\mathbbm{1}_A$ assigns to each module $X$ a submodule $tX$ of $X$.

\begin{prop}\label{subfun} There exists a one to one correspondence
\begin{align*}
   \left\{ \begin{matrix}\text{ideal torsion pairs}\\ \text{$(\mathcal{I}, \mathcal{J})$ in $\modA$}\end{matrix} \right\}&\longleftrightarrow \{\text{subfunctors of }\mathbbm{1}_A\}
\end{align*}defined as follows.
\begin{itemize}
    \item[\rm (i)] \textit{For $M\in \modA$ consider a short exact sequence
\begin{align*}
0 \longrightarrow L \xlongrightarrow{\varphi}  M \xlongrightarrow{\psi} N \longrightarrow 0
\end{align*}
with $\varphi \in \mathcal{I}$ and $\psi \in \mathcal{J}$. Then $M\mapsto \im \varphi$ defines a subfunctor of $\mathbbm{1}_A$.}
    \item[\rm (ii)] \textit{A subfunctor $t$ of $\mathbbm{1}_A$ defines an ideal torsion pair $(\mathcal{I}, \mathcal{J})$ in $\modA$ by}
    \begin{align*}
        \mathcal{I} &= \{\varphi\colon L \rightarrow M \mid \im \varphi \subseteq tM\},\\
        \mathcal{J} &= \{\psi\colon M \rightarrow N \mid tM \subseteq \ker \psi\}.
    \end{align*}
\end{itemize}
\end{prop}

\begin{proof} (i) Let $f\colon M \rightarrow M'$ be an arbitrary morphism in $\modA$. The morphism $\varphi \colon L \rightarrow M$ is a monic left $\mathcal{I}$-approximation by Remark \ref{itprem}. Similarly, let $\varphi' \colon L' \rightarrow M'$ be a monic left $\mathcal{I}$-approximation. Then $f \varphi \in \mathcal{I}$ must factor through $\varphi'$. Thus,  $f(\im \varphi) \subseteq \im \varphi'$ and the described assignment defines a subfunctor of $\mathbbm{1}_A$.

(ii) First we show that $\mathcal{I}$ and $\mathcal{J}$ are ideals. Let $\varphi\colon L \rightarrow M$ and $\varphi'\colon L' \rightarrow M$ be in $\mathcal{I}$. Then $\im \varphi, \im \varphi' \subseteq tM$. Hence,  $\im (\varphi + \varphi') \subseteq tM$ and $\varphi + \varphi'$ is contained in $\mathcal{I}$. For a morphism $f\colon X\rightarrow L $ clearly $\im \varphi f \subseteq \im \varphi \subseteq tM$ so $\varphi f \in \mathcal{I}$. For a morphism $g\colon M \rightarrow Y$, because $t$ is a subfunctor of $\mathbbm{1}_A$, the inclusion $g(tM) \subseteq tY$ holds. Thus,  $\im g \varphi =  g(\im \varphi ) \subseteq g(tM) \subseteq tY$. It follows that $g \varphi \in \mathcal{I}$ and $\mathcal{I}$ is an ideal. Similarly $\mathcal{J}$ is an ideal. By definition of $\mathcal{J}$ and $\mathcal{I}$, we have $\psi \varphi = 0$ for all $\varphi \in \mathcal{I}$ and $\psi \in \mathcal{I} = 0$. Further, every $M\in \modA$ admits the short exact sequence
\begin{align*}
    0 \longrightarrow tM \xlongrightarrow{\varphi} M \xlongrightarrow{\psi} M/tM \longrightarrow 0
\end{align*}
with $\varphi\in \mathcal{I}$ and $\psi \in \mathcal{J}$. It follows that $(\mathcal{I}, \mathcal{J})$ is an ideal torsion pair. Since the short exact sequences are also given as above in (i), it follows that the assignments are mutually inverse.
\end{proof}

\begin{rem}\label{latticecor} \rm We can consider the collection of ideal torsion pairs as a lattice by $(\mathcal{I}, \mathcal{J}) \leq (\mathcal{I}', \mathcal{J}')$ if $\mathcal{I} \subseteq \mathcal{I}'$ (or equivalently $\mathcal{J} \supseteq \mathcal{J}'$). The intersection of torsion(-free) ideals is a again a torsion(-free) ideal since they are characterized by fulfilling a closure property (Lemma \ref{charai}). Hence, the meet of the lattice is given by $(\mathcal{I}\cap \mathcal{I}', (\mathcal{I}\cap\mathcal{I}')^\perp )$ and the join by $({}^\perp (\mathcal{J} \cap  \mathcal{J}'), \mathcal{J}\cap \mathcal{J}')$ for ideal torsion pairs $(\mathcal{I}, \mathcal{J})$ and $(\mathcal{I}', \mathcal{J}')$. Because the assignment in Proposition \ref{subfun} is order-preserving, it follows that the collection of ideal torsion pairs is a complete modular lattice (as so is the lattice of subfunctors of $\mathbbm{1}_A$). Note that the lattice of torsion pairs is not necessarily modular.
\end{rem}

By the assignment in Proposition \ref{subfun}, we will analyse ideal torsion pairs through the corresponding subfunctors of $\mathbbm{1}_A$.

\begin{lem}\label{homiso}Let $(\mathcal{I}, \mathcal{J})$ be an ideal torsion pair in $\modA$ and $t$ the corresponding subfunctor of $\mathbbm{1}_A$. The natural isomorphism $\Hom(A,-)$ $\cong \mathbbm{1}_A$ restricts to an isomorphism $\mathcal{I}(A,-) \cong t$.
\end{lem}

\begin{proof} The natural isomorphism is given by $\Hom(A,M) \ni \varphi \mapsto \varphi(1)$ for $M\in \modA$. Now if $\varphi \in \mathcal{I}(A,M)$, then $\im \varphi \subseteq tM$ and in particular $\varphi(1) \in t M $. Conversly, for $x\in t M $ let $\varphi\colon A\rightarrow M$ be defined by $\varphi(1) = x$. Then $\im \varphi \subseteq t M $ and $\varphi \in \mathcal{I}(A,M)$. Hence,  the natural isomorphism restricts to $\mathcal{I}(A,-) \cong t$.
\end{proof}

Recall that $D$ denotes the duality between $\modA$ and $\modA^\op$. For an ideal $\mathcal{I}$ in $\modA$ we denote by $D \mathcal{I}$ the collection of all morphisms $\varphi$ in $\modA^\op$ isomorphic to $D\psi$ for some $\psi \in \mathcal{J}$. Clearly $D\mathcal{I}$ is an ideal of $\modA^\op$.

\begin{lem}\label{dual}Let $(\mathcal{I}, \mathcal{J})$ be an ideal torsion pair in $\modA$ and $t$ the corresponding subfunctor of $\mathbbm{1}_A$. Then $(D\mathcal{J}, D\mathcal{I})$ is an ideal torsion pair and the corresponding subfunctor of $\mathbbm{1}_{A^\op}$ is isomorphic to $D (\mathbbm{1}_A/t) D$.
\end{lem}

\begin{proof} For $D \varphi \in D \mathcal{I}$ and $D \psi \in D \mathcal{J}$ the equality $D \varphi D \psi = D (\psi \varphi) = 0$ holds. For $M\in \modA^\op$ there exists a short exact sequence
\begin{align*}
    0 \longrightarrow D (DM/tDM) \longrightarrow M \longrightarrow D tM \longrightarrow 0
\end{align*}
with the first morphism contained in $D \mathcal{J}$ and the second one contained in $D \mathcal{I}$. Thus,  $(D \mathcal{J}, D\mathcal{I})$ is an ideal torsion pair. Further, by the above short exact sequence, the corresponding subfunctor of $\mathbbm{1}_{A^\op}$ is equivalent to $D(1/t) D$.
\end{proof}

Given an additive functor $F \colon \modA \rightarrow \modA$, we can also view $F$ as a functor from $\modA$ to $\Ab$ by composing it with the forgetful functor $U \colon \modA \rightarrow \Ab$. We say $F$ is \emph{finitely presented} if $U F \in \funfp{A}$.

\begin{thm}\label{ffiff}Let $(\mathcal{I}, \mathcal{J})$ be an ideal torsion pair and $t$ the corresponding subfunctor of $\mathbbm{1}_A$. The following are equivalent:
\begin{itemize}
    \item[\rm (a)] \textit{The module $ A$ admits a left $\mathcal{I}$-approximation.}
    \item[\rm (b)] \textit{The ideal $\mathcal{I}$ is functorially finite.} 
    \item[\rm (c)] \textit{The functor $t$ is finitely presented.}
    \item[\rm (a)'] \textit{The module $D A^\op$ admits a right $\mathcal{J}$-approximation.}
    \item[\rm (b)'] \textit{The ideal $\mathcal{J}$ is functorially finite.} 
    \item[\rm (c)'] \textit{The functor $D(\mathbbm{1}_A/ t)D$ is finitely presented.}
\end{itemize}
\end{thm}

\begin{proof} (a)$\Rightarrow$(b): By Remark \ref{itprem} the ideal $\mathcal{I}$ is always contravariantly finite. It is left to show that $\mathcal{I}$ is covariantly finite. Let $M\in \modA$ and $\pi\colon A^n \rightarrow M$ an epimorphism. Given a left $\mathcal{I}$-approximation $\varphi\colon A\rightarrow C_A$ consider the following pushout diagram
\begin{equation*}
    \begin{tikzcd}
        A^n \arrow[r, "\pi"] \arrow[d, "\varphi^n", swap] & M \arrow[d, "\psi"] \\
        C_A^n \arrow[r] & P.
    \end{tikzcd}
\end{equation*}
We will show that $\psi$ is a left $\mathcal{I}$-approximation. By commutativity of the diagram $\psi \pi \in \mathcal{I}$ so $\im \psi = \im \psi \pi \subseteq tP$. Hence,  $\psi \in \mathcal{I}$. Now let $f\colon M \rightarrow X$ be a morphism in $\mathcal{I}$. Then $f \pi \in \mathcal{I}$ must factor through $\varphi^n$. By the universal property of the pushout diagram, there exists a morphism $g \colon P \rightarrow X$ with $g \psi = f$. It follows that $\psi $ is a left $\mathcal{I}$-approximation.

(b)$\Rightarrow$(c): By Lemma \ref{homiso} we can show that $\mathcal{I}(A,-) \subseteq \Hom(A,-)$ is finitely presented. Because $\Hom(A,-)$ is finitely presented, it is enough to show that $\mathcal{I}(A,-)$ is finitely generated. Let $\varphi \colon A \rightarrow C_A$ be a left $\mathcal{I}$-approximation. We show that $\im \Hom(\varphi,-) = \mathcal{I}(A,-)$. Because $\varphi\in \mathcal{I}$, the inclusion "$\subseteq$" holds. Now let $f\colon A\rightarrow X$ be a morphism in $\mathcal{I}$. Then $f$ factors through $\varphi$ and so $f\in \im \Hom(\varphi, X)$. Hence,  also "$\supseteq$" holds.

(c)$\Rightarrow$(a): By Lemma \ref{homiso} the functor $\mathcal{I}(A,-) \subseteq \Hom(A,-)$ is finitely presented. Hence, there exists an epimorphism $\Hom(M,-) \rightarrow \mathcal{I}(A,-)$ with $M\in \modA$. Now by the Yoneda lemma there exists a morphism $\varphi\colon A\rightarrow M$ with $\im \Hom(\varphi,-) = \mathcal{I}(A,-)$. We show that $\varphi$ is a left $\mathcal{I}$-approximation. First $\varphi =  \Hom(\varphi, M) (1_M)$ so $\varphi \in \mathcal{I}$. Now let $f\colon A\rightarrow X$ be a morphism in $\mathcal{I}$. Then $f\in \mathcal{I}(A,X) = \im \Hom(\varphi, X)$ so $f$ factors through $\varphi$. It follows that $\varphi$ is a left $\mathcal{I}$-approximation.

(a)'$\Leftrightarrow$(b)'$\Leftrightarrow$(c)': If $(\mathcal{I}, \mathcal{J})$ is an ideal torsion pair in $\modA$, then $(D\mathcal{J}, D\mathcal{I})$ is an ideal torsion pair by Lemma \ref{dual} with the correponding subfunctor of $\mathbbm{1}_{A^\op}$ equivalent to $D(\mathbbm{1}_A / t) D$. By duality, $ DA^\op$ admits a right $\mathcal{J}$-approximation iff $A^\op$ admits a left $D\mathcal{J}$-approximation and the ideal $\mathcal{J}$ is functorially finite iff $D\mathcal{J}$ is functorially finite. Hence,  the equivalences follow from (a)$\Leftrightarrow$(b)$\Leftrightarrow$(c).

(c)$\Leftrightarrow$(c)': Because $\mathbbm{1}_A \cong \Hom(A,-)$, the functor $t$ is finitely presented if and only if $\mathbbm{1}_A/t$ is finitely presented. By Lemma \ref{duality}, the functor $\mathbbm{1}_A/t$ is finitely presented if and only if $D(\mathbbm{1}_A /t)D$ is finitely presented.
\end{proof}

An ideal torsion pair $(\mathcal{I}, \mathcal{J})$ in $\modA$ is called \emph{functorially finite} if one of the equivalent conditions in Theorem \ref{ffiff} is fulfilled.

\begin{coro}\label{corres}The one to one correspondence
\begin{align*}
   \left\{\begin{matrix}\text{ideal torsion pairs}\\ \text{$(\mathcal{I}, \mathcal{J})$ in $\modA$}\end{matrix}\right\} &\longleftrightarrow \{\text{subfunctors of }\mathbbm{1}_A\}
\end{align*}
restricts to a one to one correspondence
\begin{align*}
    \left\{\begin{matrix}\text{functorially finite}\\ \text{ideal torsion pairs}\\ \text{$(\mathcal{I}, \mathcal{J})$ in $\modA$}\end{matrix}\right\} &\longleftrightarrow \left\{\begin{matrix}
 \text{finitely presented} \\ 
 \text{subfunctors of }\mathbbm{1}_A
\end{matrix} \right\}.
\end{align*}\\
\end{coro}

\begin{coro} If $(\mathcal{I}, \mathcal{J})$ and $(\mathcal{I}', \mathcal{J}')$ are functorially finite ideal torsion pairs, then so are $(\mathcal{I}\cap \mathcal{I}', (\mathcal{I}\cap\mathcal{I}')^\perp )$ and $({}^\perp (\mathcal{J} \cap  \mathcal{J}'), \mathcal{J}\cap \mathcal{J}')$.
\end{coro}

\begin{proof} Since $\funfp{A}$ is an abelian subcategory of $\fun{A}$, it follows that the collection of finitely presented subfunctors of $\mathbbm{1}_A$ is a sublattice of the lattice of subfunctors of $\mathbbm{1}_A$. Hence, by Corollary \ref{corres}, the lattice of functorially finite ideal torsion pairs is a sublattice of the lattice of ideal torsion pairs. Now the claim follows by Remark \ref{latticecor}.
\end{proof}

We introduce a monoidal structure on pairs $s \leq t$ of subfunctor of $\mathbbm{1}_A$ as a method for producing new ideal torsion pairs via the correspondence in Proposition \ref{subfun}. Let $s\leq t$ and $s' \leq t'$ be subfunctors of $\mathbbm{1}_A$. Then the composition $t (t'/s') $ is a subfunctor of $t'/s'$ and thus isomorphic to $t''/s'$ for some $t''\leq t'$. Similarly, the composition $s(t'/s')$ is a subfunctor of $t'/s'$ and thus isomorphic to $s''/s'$ for some $s'' \leq t'$. Hence,  
\begin{align*}
    (t/s)(t'/s') = (t''/s')/(s''/s') \cong t''/s''.
\end{align*}
In this way, composing factors $t/s$ of subfunctors $s\leq t$ of $\mathbbm{1}_A$ yields a monoidal structure with the neutral element $\mathbbm{1}_A$. The following lemma shows that the monoidal structure restricts to finitely presented functors.

\begin{lem}\label{rest} Let $F\colon \modA \rightarrow \modA$ be an additive finitely presented functor and $G\in \funfp{A}$. Then $GF \in \funfp{A}$. 
\end{lem}

\begin{proof}
    We can uniquely extend $F$ and $G$ to functors $\overline{F}\colon \textnormal{Mod}\,A \rightarrow \textnormal{Mod}\,A$ and $\overline{G}\colon \textnormal{Mod}\,A \rightarrow \Ab$ such that $\overline{F}$ and $\overline{G}$ commute with direct limits and $\overline{F}, \overline{G}$ coincide with $F,G$ on $\modA$ respectively as follows. For $M\in \ModA$ we have $M = \varinjlim M_i$ for $M_i \in \modA$ and we set 
    \begin{align*}
        \overline{F}(M) = \varinjlim F(M_i), \qquad \overline{G}(M) = \varinjlim G(M_i).
    \end{align*}
   Then $\overline{F}$ and $\overline{G}$ have the desired properties (see for example \cite{Prest2}). Because $F$ and $G$ are finitely presented, it follows from \cite[Theorem 9,1]{Krause} that $\overline{F}$ and $\overline{G}$ commute with products. Now $\overline{G}\,\overline{F}$ is the unique extension, in the above sense, of $GF$ as the functors coincide on $\modA$ and $\overline{G}\,\overline{F}$ commutes with direct limits, since $\overline{F}$ and $\overline{G}$ commute with direct limits. Further, $\overline{G}\,\overline{F}$ commutes with products, since so do $\overline{G}$ and $\overline{F}$. By \cite[Theorem 9,1]{Krause} it follows that $GF$ is finitely presented.
\end{proof}

We continue by investigating special cases of the monoidal structure on pairs $s\leq t$ of subfunctors of $\mathbbm{1}_A$, namely $s = 0$ and arbitrary $t$, as well as $t = \mathbbm{1}_A$ and arbitrary $s$. In particular, we are interested in the corresponding ideal torsion pairs. To describe them, we introduce two operations on ideals $\mathcal{I}$ and $\mathcal{I}'$. First, we denote by $\mathcal{I}' \mathcal{I}$ the collection all morphisms $\varphi'\varphi$ with $\varphi \in \mathcal{I}$ and $\varphi \in \mathcal{I}'$ (if the composition is defined). It is easy to check that $\mathcal{I}' \mathcal{I}$ is again an ideal. For the other operation, we say that a morphism $f \colon X \rightarrow Y$ is an \emph{extension} of a morphism $\varphi'$ by a morphism $\varphi$ if there exists a commuative diagram of morphisms
\begin{equation*}
    \begin{tikzcd}
       & & X \arrow[d, "h", swap] \arrow[dr, "\varphi' "] & & \\
       0 \arrow[r] & L \arrow[dr, "\varphi", swap] \arrow[r] & M \arrow[d, "g"] \arrow[r] & N \arrow[r] &0 \\
        & & Y & &     
    \end{tikzcd}
\end{equation*}
with $f = g h$, where the horizontal sequence is exact. Now we denote by $\mathcal{I} \diamond \mathcal{I}'$ the collection of all morphisms $f$ that are an extension of a morphism $\varphi' \in \mathcal{I}'$ by a morphism $\varphi \in \mathcal{I}$.

The above definition of extensions of morphisms were introduced in \cite{Fu2}. The following result was shown by the first three authors in \cite{Zhu}.

\begin{prop}\label{scase} Let $(\mathcal{I}, \mathcal{J})$, $(\mathcal{I}', \mathcal{J}')$ be ideal torsion pairs in $\modA$ and $t, t'$ the corresponding subfunctors of $\mathbbm{1}_A$.
\begin{itemize}
    \item[\rm (a)] The functor $t t'$ corresponds to the ideal torsion pair $(\mathcal{I}'\mathcal{I}, \mathcal{J} \diamond \mathcal{J}' )$.
    \item[\rm (b)] The functor $t''$, defined by $\mathbbm{1}_A /t'' = (\mathbbm{1}_A/t)(\mathbbm{1}_A/t')$, corresponds to the ideal torsion pair $(\mathcal{I} \diamond \mathcal{I}', \mathcal{J}' \mathcal{J})$.
\end{itemize}
\end{prop}

\begin{proof} (a) For $M\in \modA$ the inclusion $t'M \rightarrow M$ is in $\mathcal{I}'$ and the inclusion $tt'M \rightarrow t' M$ is in $\mathcal{I}$. Thus, the inclusion $tt'M \rightarrow M$ is in $\mathcal{I}' \mathcal{I}$. Now an arbitrary morphism in $\mathcal{I}' \mathcal{I}$ factors as $\varphi' \varphi$ for $\varphi \colon L \rightarrow N$ in $\mathcal{I}$ and $\varphi' \colon N \rightarrow M$ in $\mathcal{I}'$. Further, $\varphi'$ factors as $N\rightarrow t'M \rightarrow M$ and the composition $L \rightarrow N \rightarrow t'M$ factors as $L \rightarrow t t'M \rightarrow t'M$. We conclude that $\im \varphi' \varphi \subseteq t t'M$ and the torsion ideal corresponding to $tt'$ must be equal to $\mathcal{I}' \mathcal{I}$. It is left to show that $(\mathcal{I}' \mathcal{I})^\perp = \mathcal{J} \diamond \mathcal{J}'$.

For $f\colon X \rightarrow Y$ in $\mathcal{J} \diamond \mathcal{J}'$ there exists a commutative diagram of morphisms
\begin{equation*}
    \begin{tikzcd}
       & & X \arrow[d, "h", swap] \arrow[dr, "\psi' "] & & \\
       0 \arrow[r] & L \arrow[dr, "\psi", swap] \arrow[r, "\iota"] & M \arrow[d, "g"] \arrow[r, "\pi", swap] & N \arrow[r] &0 \\
        & & Y & &     
    \end{tikzcd}
\end{equation*}
with $f = gh$ and $\psi \in \mathcal{J}, \psi' \in \mathcal{J}'$, where the horizontal sequence is exact. Now for $\varphi\in \mathcal{I},\varphi' \in \mathcal{I}'$ we have $\pi h \varphi' = \psi' \varphi' = 0$. Hence,  $h \varphi'$ must factor through $\iota$, that is $h \varphi' = \iota \alpha$. Thus,   $f \varphi' \varphi = gh \varphi' \varphi = g \iota \alpha \varphi = \psi \alpha \varphi = 0$ and so $f \in (\mathcal{I}' \mathcal{I})^\perp$.

Let $f\colon X \rightarrow Y$ in $(\mathcal{I}' \mathcal{I})^\perp$ and consider the commutative diagram  
\begin{equation*}
    \begin{tikzcd}
       & & X \arrow[d, equals] \arrow[dr] & & \\
       0 \arrow[r] & t'X \arrow[dr] \arrow[r] & X \arrow[d] \arrow[r] & X/t'X \arrow[r] &0 \\
        & & X/tt'X. & &     
    \end{tikzcd}
\end{equation*}
Since $t' X \rightarrow t'X/tt'X$ is contained in $\mathcal{J}$ and $X\rightarrow X/t'X$ in $\mathcal{J}'$, we conclude that $X \rightarrow X/tt'X$ is contained in $\mathcal{J}\diamond \mathcal{J}'$. Because $tt'X \rightarrow X$ is contained in $\mathcal{I}' \mathcal{I}$, the morphisms $f$ factors through $X \rightarrow X/tt'X$. Thus,   also $f\in \mathcal{J} \diamond \mathcal{J}'$.

(b) Similar to (a).
\end{proof}

\begin{coro}\label{ext} If $(\mathcal{I}, \mathcal{J})$ and $(\mathcal{I}', \mathcal{J}')$ are functorially finite ideal torsion pairs, then so are $(\mathcal{I}'\mathcal{I}, \mathcal{J} \diamond \mathcal{J}' )$ and $(\mathcal{I} \diamond \mathcal{I}', \mathcal{J}' \mathcal{J})$.
\end{coro}

\begin{proof} Let $t$ and $t'$ be the subfunctors of $\mathbbm{1}_A$ corresponding to the ideal torsion pairs $(\mathcal{I}, \mathcal{J})$ and $(\mathcal{I}', \mathcal{J}')$. By Corollary \ref{corres} the functors are finitely presented. Now $tt'$ and $t''$, defined by $\mathbbm{1}_A /t'' = (\mathbbm{1}_A/t)(\mathbbm{1}_A/t')$, correspond to $(\mathcal{I}'\mathcal{I}, \mathcal{J} \diamond \mathcal{J}' )$ and $(\mathcal{I} \diamond \mathcal{I}', \mathcal{J}' \mathcal{J})$ respectively by Proposition \ref{scase}. By Lemma \ref{rest} the functors $tt'$ and $t''$ are finitely presented. Hence,  the corresponding ideal torsion pairs are functorially finite.
\end{proof}

For later purposes, the above result will be important for producing new functorially finite ideal torsion pairs. In particular, the involved ideals will be related to subcategories of $\modA$. We check that the notion of extensions of morphisms behaves well with the notion of extensions of modules. For full additive subcategories $\mathcal{C}, \mathcal{C}'$ of $\modA$, we denote by $\mathcal{C} \diamond \mathcal{C}'$ the collection of all modules $M\in \modA$ such that there exists a short exact sequence
\begin{align*}
    0 \longrightarrow L \longrightarrow M \longrightarrow N \longrightarrow 0
\end{align*}
with $L \in \mathcal{C}$ and $N \in \mathcal{C}'$.

\begin{lem}\label{objext} Let $\mathcal{C}$ and $\mathcal{C}'$ be full additive subcategories of $\modA$. Then 
\begin{align*}
\langle \mathcal{C} \rangle \diamond \langle \mathcal{C}' \rangle = \langle \mathcal{C} \diamond \mathcal{C}' \rangle.    
\end{align*}
\end{lem}

\begin{proof} Let $f \colon X \rightarrow Y$ be in $\langle \mathcal{C} \diamond \mathcal{C}' \rangle$. Then $f$ factors as $gh$ with $g$ starting in $M \in \mathcal{C} \diamond \mathcal{C}'$. Hence,  there exists a short exact sequence
\begin{align*}
    0 \longrightarrow L \longrightarrow M \longrightarrow N \longrightarrow 0
\end{align*}
with $L\in \mathcal{C}$ and $N\in \mathcal{C}'$. We extend this to a commutative diagram
\begin{equation*}
    \begin{tikzcd}
       & & X \arrow[d, "h", swap] \arrow[dr, "\varphi' "] & & \\
       0 \arrow[r] & L \arrow[dr, "\varphi", swap] \arrow[r] & M \arrow[d, "g"] \arrow[r] & N \arrow[r] &0 \\
        & & Y & &     
    \end{tikzcd}
\end{equation*}
with $\varphi \in \langle \mathcal{C} \rangle$ and $\varphi' \in \langle \mathcal{C}' \rangle$. Thus,   $f \in \langle \mathcal{C} \rangle \diamond \langle \mathcal{C}' \rangle$.

Let $f\colon X\rightarrow Y$ be in $\langle \mathcal{C} \rangle \diamond \langle \mathcal{C}' \rangle$. Then there exists a commutative diagram
\begin{equation*}
    \begin{tikzcd}
       & & X \arrow[d, "h", swap] \arrow[dr, "\varphi' "] & & \\
       0 \arrow[r] & L \arrow[dr, "\varphi", swap] \arrow[r, "\iota"] & M \arrow[d, "g"] \arrow[r, "\pi"] & N \arrow[r] &0 \\
        & & Y & &     
    \end{tikzcd}
\end{equation*}
with exact middle row, $f = gh$ and $\varphi \in \langle \mathcal{C} \rangle$, $\varphi' \in \langle \mathcal{C}' \rangle$. In particular $\pi h = \varphi'$ factors through some $C'\in \mathcal{C}'$ and taking a suitable pullback $P$, we obtain a commutative diagram
\begin{equation*}
    \begin{tikzcd}
    & & X \arrow[d, "\alpha"] \\
       0 \arrow[r] & L \arrow[d, equals] \arrow[r, "\beta"] & P \arrow[d, "h'"] \arrow[r] & C' \arrow[d] \arrow[r] & 0 \\
       0 \arrow[r] & L  \arrow[r, "\iota"] & M \arrow[r, "\pi"] & N \arrow[r] &0 
    \end{tikzcd}
\end{equation*}
with $h = h' \alpha$. Now $gh' \beta = g\iota = \varphi$ factors through some $C\in \mathcal{C}$. Taking a suitable pushout $Q$, we obtain a commutative diagram
\begin{equation*}
    \begin{tikzcd}
       0 \arrow[r] & L \arrow[d] \arrow[r, "\beta"] & P \arrow[d, "g'"] \arrow[r] & C' \arrow[d, equals] \arrow[r] & 0 \\
       0 \arrow[r] & C  \arrow[r] & Q \arrow[d, "\gamma"] \arrow[r] & C' \arrow[r] &0 \\
        & &  M & &
    \end{tikzcd}
\end{equation*}
with $\gamma g' = gh'$. Hence,  $f = gh = g h' \alpha = \gamma g' h' \alpha$ factors through $Q \in \mathcal{C} \diamond \mathcal{C}'$. We conclude that $\langle \mathcal{C} \rangle \diamond \langle \mathcal{C}' \rangle = \langle \mathcal{C} \diamond \mathcal{C}' \rangle.$
\end{proof}

\section{Ideals determined by objects}

Motivated by Auslander's concept of morphisms determined by objects (see \cite{Auslander1} and \cite{Auslander2}), we introduce the notion of ideals determined by objects. This will offer a different approach to torsion ideals and torsion-free ideals.

Let $\mathcal{I}$ be an ideal of $\modA$ and $C\in \modA$. Then $\mathcal{I}$ is \emph{right $C$-determined} if $\varphi f \in \mathcal{I}$ for all $f$ starting in $C$ (such that the composition is defined) already implies $\varphi \in \mathcal{I}$. Dually $\mathcal{I}$ is \emph{left $C$-determined} if $f \varphi \in \mathcal{I}$ for all $f$ ending in $C$ already implies $\varphi \in \mathcal{I}$.

\begin{prop}\label{adet} \begin{itemize}
    \item[\rm (a)] An ideal $\mathcal{I}$ is a torsion ideal if and only if $\mathcal{I}$ is right $A$-determined.
    \item[\rm (b)] An ideal $\mathcal{J}$ is a torsion-free ideal if and only if $\mathcal{J}$ is left $D A^\op$-determined.
\end{itemize}
\end{prop}

\begin{proof} (a) Let $\mathcal{I}$ be a torsion ideal and $\varphi$ a morphism such that $\varphi f \in \mathcal{I}$ for all $f$ starting in $A$. We can choose $f$ as an epimorphism starting in $A^n$ such that $\varphi f\in \mathcal{I}$. By Lemma \ref{charai} it follows that $\varphi \in \mathcal{I}$. Hence,  $\mathcal{I}$ is right $A$-determined. 

Let $\mathcal{I}$ be right $A$-determined and $\varphi$ a morphism such that $\varphi f \in \mathcal{I}$ for an epimorphism $f$. Then for an arbitrary morphism $f'$ starting in $A$ (such that $\varphi f'$ is defined), $f'$ factors through the epimorphism $f$. Thus,   $f' \varphi \in \mathcal{I}$ and because $\mathcal{I}$ is right $A$-determined, we conclude that $\varphi \in \mathcal{I}$. By Lemma \ref{charai} it follows that $\mathcal{I}$ is a torsion ideal.

(b) Similar to (a).
\end{proof}

\begin{rem}\label{smallt} \rm For an ideal $\mathcal{I}$ of $\modA$ let $\textnormal{I}(\mathcal{I})$ denote the smallest torsion ideal containing $\mathcal{I}$. In Section 2 we have indirectly seen two ways to describe $\textnormal{I}(\mathcal{I})$. Namely $\textnormal{I}(\mathcal{I}) = {}^\perp (\mathcal{I}^\perp)$ by Proposition \ref{ortho} and $\textnormal{I}(\mathcal{I})$ equals the collection of all morphisms $\varphi$ such that $\varphi f \in \mathcal{I}$ for all epimorphisms $f$ by Lemma \ref{charai}. Now Proposition \ref{adet} offers the most useful description for our purposes: The ideal $\textnormal{I}(\mathcal{I})$ equals the collection of all morphisms $\varphi$ such that $\varphi f \in \mathcal{I}$ for all morphisms $f$ starting in $A$. In particular $\textnormal{I}(\mathcal{I})$ is uniquely determined by $\mathcal{I}(A,-)$. 

A similar description holds for the smallest torsion-free ideal containing $\mathcal{I}$, denoted by $\textnormal{{J}}(\mathcal{I})$: The ideal $\textnormal{{J}}(\mathcal{I})$ equals the collection of all morphisms $\psi$ such that $f \psi \in \mathcal{I}$ for all morphisms $f$ ending in $DA^\op$.
\end{rem}

Let $\mathcal{I}$ be an ideal. We would like to know when $\textnormal{I}(\mathcal{I})$ or $\textnormal{J}(\mathcal{I})$ is functorially finite depending on $\mathcal{I}$. This is fully answered by the following result.

\begin{coro}\label{coro} Let $\mathcal{I}$ be an ideal of $\modA$. 
\begin{itemize}
    \item[\rm (a)] The torsion ideal ${\rm I}(\mathcal{I})$ is functorially finite if and only if $A$ admits a left $\mathcal{I}$-approximation.
    \item[\rm (b)] The torsion-free ideal ${\rm J}(\mathcal{I})$ is functorially finite if and only if $DA^\op$ admits a right $\mathcal{I}$-approximation.
\end{itemize}
\end{coro}

\begin{proof} (a) By Remark \ref{smallt} we have $\mathcal{I}(A,-) = {\rm I}(\mathcal{I})(A,-)$. Thus,   $A$ admits a left $\mathcal{I}$-approximation if and only if $A$ admits a left ${\rm I}(\mathcal{I})$-approximation. Now the claim follows by Theorem \ref{ffiff}.

(b) Similar to (a).
\end{proof}

Next, we investigate when a torsion ideal is also left $C$-determined and when a torsion-free ideal is also right $C$-determined.

\begin{lem}\label{deter} Let $(\mathcal{I}, \mathcal{J})$ be an ideal torsion pair.
\begin{itemize}
    \item[\rm (a)] The ideal $\mathcal{I}$ is left $C$-determined for some $C\in \modA$ if and only if $\mathcal{J}$ is functorially finite.
    \item[\rm (b)] The ideal $\mathcal{J}$ is right $C$-determined for some $C\in \modA$ if and only if $\mathcal{I}$ is functorially finite.
\end{itemize}
\end{lem}

\begin{proof} (a) Let $\mathcal{I}$ be left $C$-determined. We denote by $\mathcal{J}_C$ the ideal consisting of all morphisms $\psi f$ with $\psi \in \mathcal{J}$ starting in $C^n$ for some $n>0$ and $f$ arbitrary. Clearly $\mathcal{J}_C \subseteq \mathcal{J}$ and thus ${\rm J}(\mathcal{J}_C) \subseteq \mathcal{J}$. To deduce equality, we show ${}^\perp \mathcal{J}_C \subseteq \mathcal{I}$. Let $\varphi$ such that $\psi f \varphi = 0$ for all $\psi \in \mathcal{J}$ starting in $C^n$ and arbitrary $f$. Then $ f \varphi \in {}^\perp \mathcal{J} = \mathcal{I}$ and because $\mathcal{I}$ is left $C$-determined, also $\varphi \in \mathcal{I}$. We conclude that ${}^\perp \mathcal{J}_C \subseteq \mathcal{I}$. Hence,  ${\rm J}(\mathcal{J}_C) = \mathcal{J}$. By Corollary \ref{coro} it suffices to show that $DA^\op$ admits a right $\mathcal{J}_C$-approximation for $\mathcal{J}$ to be functorially finite. Let $\psi_i, 1\leq i \leq m$ generate $\mathcal{J}(C, DA^\op)$ as a $k$-module. Then all $\psi_i$ induce a morphism $\psi \colon C^m \rightarrow  DA^\op$. Clearly $\psi \in \mathcal{J}_C$. Now let $\alpha\colon M \rightarrow DA^\op$ be in $\mathcal{J}_C$, so $\alpha = \beta f$ for some $f$ and for some $\beta \in \mathcal{J}$ starting in $C^n$ with $n \in \mathbbm{N}$. Then by construction of $\psi$, the morphism $\beta $ factors through $\psi$. Hence,  $\alpha$ factors through $\psi$ and $\psi$ is a right $\mathcal{J}_C$-approximation.

Let $\mathcal{J}$ be functorially finite and $\psi\colon C\rightarrow DA^\op$ a right $\mathcal{J}$-approximation. Further, let $\varphi \colon M \rightarrow N$ be a morphism such that $f \varphi \in \mathcal{I}$ for all morphisms $f$ ending in $C$. We show $\varphi \in {}^\perp \mathcal{J} = \mathcal{I}$. Let $\alpha\colon N \rightarrow L $ in $\mathcal{J}$ and $\iota \colon L \rightarrow (DA^\op)^n$ a monomorphism. Then $\iota \alpha$ factors through $\psi^n$, that is $\iota \alpha = \psi^k f$. Now $f \varphi \in \mathcal{I}$ implies $\iota 
\alpha \varphi = \psi^k f \varphi = 0$. Thus,   $\alpha \varphi = 0$ and we conclude that $\varphi \in {}^\perp \mathcal{J} = \mathcal{I}$. Hence,  $\mathcal{I}$ is left $C$-determined.

(b) Similar to (a).
\end{proof}

It is interesting to note that, for an ideal torsion pair $(\mathcal{I}, \mathcal{J})$, the property of one of the ideals to be determined by an object on its non-trivial side is equivalent to the other ideal being functorially finite. Now the property of being functorially finite is equivalent for both of the ideals by Theorem \ref{ffiff}. Hence, we can describe functorially finite torsion ideals by the objects that they are determined by, plus some additional information. This is done by the following result.

\begin{thm}\label{det}
\begin{itemize}
    \item[\rm (a)] In $\modA$ we have an equality
    \begin{align*}
        \left\{\begin{matrix}\text{functorially finite}\\ \text{torsion ideals}\end{matrix}\right\} = \bigcup_{C \in \modA } \left\{\begin{matrix}\text{left $C$-determined} \\ \text{torsion ideals}\end{matrix}\right\}.
    \end{align*}
    \item[\rm (b)] \textit{For $C\in \modA$ there exists a one to one correspondence
    \begin{align*}
        \left\{\begin{matrix}\text{left $C$-determined} \\ \text{torsion ideals}\end{matrix}\right\} \longleftrightarrow \left\{\begin{matrix}
 \text{bi-submodules of} \\ 
 \text{ ${}_AC_{\textnormal{End}_A(C)^{\textnormal{op}}}$ }
\end{matrix} \right\}
    \end{align*}
    via $\mathcal{I} \mapsto tC$, where $t$ denotes the subfunctor of $\mathbbm{1}_A$ corresponding to $\mathcal{I}$.}
\end{itemize}
\end{thm}

\begin{proof} (a) This follows by Lemma \ref{deter} and Theorem \ref{ffiff}.

(b) By the functoriality of $t$, the submodule $tC$ of $C$ is invariant under $\textnormal{End}_A(C)$. Hence,  $tC$ is a submodule of the bimodule ${}_AC_{\textnormal{End}_A(C)^{\textnormal{op}}}$. To show that the assignment $C \mapsto tC$ is one to one, we construct an inverse assignment: For a submodule $X$ of ${}_AC_{\textnormal{End}_A(C)^{\textnormal{op}}}$ let $\mathcal{I}_X$ be the collection of all morphisms $\varphi\colon M \rightarrow N$ such that for all $f\colon A\rightarrow M$ and $g\colon N \rightarrow C$ we have $g \varphi f (1) \in X$. Clearly $\mathcal{I}_X$ is an ideal and right $A$-determined as well as left $C$-determined by construction. By Proposition \ref{adet} the ideal $\mathcal{I}_X$ is a left $C$-determined torsion ideal. Further, $X\rightarrow C$ is contained in $\mathcal{I}_X$ and $C \rightarrow C/X$ in $\mathcal{I}_X ^\perp$. Hence, if $t$ denotes the subfunctor of $\mathbbm{1}_A$ corresponding to $\mathcal{I}_X$, then $tC = X$. It is left to show $\mathcal{I}_{tC} = \mathcal{I}$, where $t$ denotes the subfunctor of $\mathbbm{1}_A$ corresponding to $\mathcal{I}$. If $\varphi \colon M \rightarrow N$ is in $\mathcal{I}$, then for all $f \colon A \rightarrow M$ and $g\colon N \rightarrow C$ we have $\im g \varphi f \subseteq tC$. Thus,   $g\varphi f (1) \in tC$ and $\varphi \in \mathcal{I}_{tC}$. On the contrary, if $\varphi \colon M \rightarrow N$ is in $\mathcal{I}_{tC}$, then for all $f\colon A \rightarrow M$ and $g\colon N \rightarrow C$ we have $ g \varphi f(1) \in tC $. Thus,   always $\im g \varphi \subseteq tC$ and $g \varphi \in \mathcal{I}$. Because $\mathcal{I}$ is left $C$-determined, we conclude that $\varphi \in \mathcal{I}$, so $\mathcal{I}_{tC} = \mathcal{I}$.
\end{proof}

We proceed by giving a connection with the classical notion of morphisms determined by objects (see Section 1).

\begin{lem}\label{conn}
\begin{itemize}
    \item[\rm (a)] Let $\mathcal{I}$ be a functorially finite torsion ideal and $\varphi\colon A \rightarrow M$ a left $\mathcal{I}$-approximation. Then $\mathcal{I}$ is left $C$-determined if and only if $\varphi$ is left $C$-determined. 
    \item[\rm (b)] Let $\mathcal{J}$ be a functorially finite torsion-free ideal and $\psi \colon N \rightarrow DA^\op$ a right $\mathcal{J}$-approximation. Then $\mathcal{J}$ is right $C$-determined if and only if $\psi$ is right $C$-determined.
\end{itemize}
\end{lem}

\begin{proof} (a) Let $\mathcal{I}$ be left $C$-determined and $\varphi' \colon A \rightarrow M'$ a morphism such that $f \varphi'$ factors through $\varphi$ for all $f \colon M' \rightarrow C$. Then always $f \varphi' \in \mathcal{I}$ and because $\mathcal{I}$ is left $C$-determined, we conclude that $\varphi'\in \mathcal{I}$. Hence,  $\varphi'$ factors through the left $\mathcal{I}$-approximation $\varphi$. It follows that $\varphi$ is left $C$-determined.

Let $\varphi$ be left $C$-determined and $\varphi' \colon M' \rightarrow N$ a morphism such that $f \varphi' \in \mathcal{I}$ for all $f \colon N \rightarrow C$. Then for all $g \colon A \rightarrow M'$ the morphism $f \varphi' g \in \mathcal{I}$ factors through $\varphi$. Because $f$ is arbitrary and $\varphi$ left $C$-determined, it follows that $\varphi' g$ factors through $\varphi$. Thus,   $\varphi' g\in \mathcal{I}$ and because $\mathcal{I}$ is right $A$-determined by Proposition \ref{adet}, we conclude that $\varphi \in \mathcal{I}$. Hence,  $\mathcal{I}$ is left $C$-determined.

(b) Similar to (a).
\end{proof}

For a functorially finite ideal torsion pair $(\mathcal{I}, \mathcal{J})$ we have already seen a way to find $C \in \modA$ such that $\mathcal{I}$ is left $C$-determined. Namely, we can choose $C$ by a right $\mathcal{J}$-approximation $C \rightarrow DA^\op$ (see the proof of Lemma \ref{deter}). The following result shows a way to find $C \in \modA$, only considering the ideal $\mathcal{I}$.

\begin{coro}
\begin{itemize}
    \item[\rm (a)] Let $\mathcal{I}$ be a functorially finite torsion ideal, $\varphi \colon A \rightarrow M$ a left $\mathcal{I}$-approximation, $K$ the kernel of $\varphi$ and $Q$ its cokernel. Then $\mathcal{I}$ is left $C$-determined for $C = \tau Q \oplus I$, where $\tau$ denotes the Auslander-Reiten translation and $I$ the injective hull of the top of $K$.
    \item[\rm (b)] Let $\mathcal{J}$ be a functorially finite torsion-free ideal, $\psi \colon N \rightarrow DA^\op $ a right $\mathcal{J}$-approximation, $K$ the kernel of $\psi$ and $Q$ its cokernel. Then $\mathcal{J}$ is right $C$-determined for $C = \tau^{-} K \oplus P$, where $\tau^{-}$ denotes the inverse of the Auslander-Reiten translation and $P$ the projective hull of the socle of $K$. 
\end{itemize}
\end{coro}

\begin{proof} This follows by Lemma \ref{conn} and Proposition \ref{ar}.
\end{proof}

\section{Subcategories related to ideal torsion pairs}

We start by discussing the obvious subcategories of $\modA$ which can relate to ideal torsion pairs: torsion classes and torsion-free classes.

\begin{rem}\label{torsion}\rm Let $(\mathcal{T}, \mathcal{F})$ be a torsion pair. Then $\Hom(M,N) = 0$ for all $M\in \mathcal{T}$ and $N\in \mathcal{F}$ implies $\psi \varphi = 0$ for all $\varphi\in \langle \mathcal{T} \rangle$ and $\psi \in \langle \mathcal{F} \rangle$. Further, the short exact sequence
\begin{align*}
    0 \longrightarrow L \xlongrightarrow{\varphi} M \xlongrightarrow{\psi} N \longrightarrow 0
\end{align*}
with $L\in \mathcal{T}$ and $N\in \mathcal{F}$ also fulfills $\varphi \in \langle \mathcal{T} \rangle$ and $\psi \in \langle \mathcal{F} \rangle$. Hence,  $(\langle \mathcal{T} \rangle, \langle \mathcal{F} \rangle)$ is an ideal torsion pair. Thus,   $\langle \mathcal{T} \rangle$ is a torsion ideal and $\langle \mathcal{F} \rangle$ a torsion-free ideal.
\end{rem}

Applying theory of ideal torsion pairs to torsion pairs, we obtain the following well-known result \cite{Smalo}.

\begin{coro} Let $(\mathcal{T}, \mathcal{F})$ be a torsion pair. Then $\mathcal{T}$ is functorially finite if and only if $\mathcal{F}$ is functorially finite.
\end{coro}

\begin{proof} By Lemma \ref{iff} we can check that $\langle \mathcal{T} \rangle$ is functorially finite if and only if $\langle \mathcal{F} \rangle$ is functorially finite. This is the case by Remark \ref{torsion} and Theorem \ref{ffiff}.
\end{proof}

For an ideal $\mathcal{I}$ let $\ob \mathcal{I}$ be the collection of all $M\in \modA$ with $1_M \in \mathcal{I}$ or equivalently $tM = M$. A full additive subcategory $\mathcal{C}$ of $\modA$ is a \textit{mono-closed} (\textit{epi-closed}) \textit{class}, if $M\in \mathcal{C}$ and $N\leq M$ implies $N \in \mathcal{C}$ (respectively $M/N \in \mathcal{C}$). As it turns out, those classes are precisely the ones that arise from ideal torsion pairs.

\begin{lem}\label{prooop}\begin{itemize}
    \item[\rm (a)] If $\mathcal{C}$ is an epi-closed class, then $\langle \mathcal{C}\rangle $ is a torsion ideal. If $\mathcal{I}$ is a torsion ideal, then $\ob \mathcal{I}$ is an epi-closed class.
    \item[\rm (b)] If $\mathcal{C}$ is a mono-closed class, then $\langle \mathcal{C}\rangle $ is a torsion-free ideal. If $\mathcal{J}$ is a torsion-free ideal, then $\ob \mathcal{J}$ is a mono-closed class.
\end{itemize}
\end{lem}

\begin{proof} (a) Let $\varphi$ be a morphism and $f$ an epimorphism with $\varphi f \in \langle \mathcal{C} \rangle$. Then $\varphi f$ factors as $\varphi_1 \varphi_2$ with $\varphi_1$ starting in $C\in \mathcal{C}$. Because $\mathcal{C}$ is an epi-closed class, the image of $\varphi_1$ is contained in $\mathcal{C}$. Now $\im \varphi = \im \varphi f \subseteq \im \varphi_1 $ implies $\varphi \in \langle \mathcal{C}\rangle$. By Lemma \ref{charai} it follows that $\langle \mathcal{C} \rangle$ is a torsion ideal.

Let $M\in \ob \mathcal{I}$ and $N \leq M$. Consider the projection $\pi\colon M \rightarrow M/N$. Then $\pi = \pi 1_M \in \mathcal{I}$. Further, $\pi = 1_{M/N} \pi$. By Lemma \ref{charai} it follows that $1_{M/N} \in \mathcal{I}$. Hence,  $M/N \in \ob \mathcal{I}$ and $\ob \mathcal{I}$ is an epi-closed class.  

(b) Similar to (a).
\end{proof}

\begin{rem}\label{remrem} \rm By Lemma \ref{tclosure} every torsion(-free) class in $\modA$ is an epi-closed (mono-closed) class. On the contrary, an epi-closed (mono-closed) class in $\modA$ is a torsion(-free) class if and only if it is closed under extensions. In particular this shows with Lemma \ref{prooop} that, in general, there are much more ideal torsion pairs then torsion pairs.
\end{rem}


Let $(\mathcal{I}, \mathcal{J})$ be an ideal torsion pair of $\modA$. Then $\langle \ob \mathcal{I} \rangle \subseteq \mathcal{I}$ is a torsion ideal and $\langle \ob \mathcal{J} \rangle \subseteq \mathcal{J}$ a torsion-free ideal by Lemma \ref{prooop}. We show a criteria for when $\mathcal{I} = \langle \ob \mathcal{I} \rangle$ or $\mathcal{J} = \langle \ob \mathcal{J} \rangle$.

\begin{prop}\label{objfun} Let $(\mathcal{I}, \mathcal{J})$ be an ideal torsion pair and $t$ the corresponding subfunctor of $\mathbbm{1}_A$.
\begin{itemize}
    \item[\rm (a)] The equality $\mathcal{I} = \langle \ob \mathcal{I} \rangle$ holds if and only if $t^2 = t$. In that case $tM \in \ob \mathcal{I}$ for all $M\in \modA$.
    \item[\rm (b)] The equality $\mathcal{J} = \langle \ob \mathcal{J}\rangle$ holds if and only if $(1/t)^2 = 1/t$. In that case $M/tM \in \ob \mathcal{J}$ for all $M\in \modA$.
    \item[\rm (c)] The equalities $\mathcal{I} = \langle \ob \mathcal{I} \rangle$ and $\mathcal{J} = \langle \ob \mathcal{J} \rangle$ hold if and only if $(\ob \mathcal{I}, \ob \mathcal{J})$ is a torsion pair. 
\end{itemize}

\begin{proof} (a) By Proposition \ref{scase} the torsion ideal $\mathcal{I}^2$ corresponds to $t^2$. Hence, if $\mathcal{I} = \langle \ob \mathcal{I} \rangle$, then $\mathcal{I}^2 = \mathcal{I}$ and $t^2 = t$. Let $M\in \modA$ and consider the canonical inclusion $\iota \colon t^2 M \rightarrow tM$. Then $\iota \in \mathcal{I}$. If $t^2 = t$, then $\iota = 1_{tM}$ and $tM\in \ob \mathcal{I}$. In particular $\mathcal{I} = \langle \ob \mathcal{I} \rangle$, as every morphism $L\rightarrow M$ in $\mathcal{I}$ factors through $tM$.

(b) Similar to (a).

(c) If $(\ob \mathcal{I}, \ob \mathcal{J})$ is a torsion pair, then $(\langle \ob \mathcal{I} \rangle, \langle \ob \mathcal{I} \rangle)$ is an ideal torsion pair by Remark \ref{torsion}. Since $\langle \ob \mathcal{I} \rangle \subseteq \mathcal{I}$ and $\langle \ob \mathcal{J} \rangle \subseteq \mathcal{J}$, it follows that $(\mathcal{I}, \mathcal{J})$ equals $(\langle \ob \mathcal{I} \rangle, \langle \ob \mathcal{I} \rangle)$ by Proposition \ref{ortho}. On the other hand, if $\mathcal{I} = \langle \ob \mathcal{I} \rangle$ and $\mathcal{J} = \langle \ob \mathcal{J} \rangle$, then $\Hom(\ob \mathcal{I} , \ob \mathcal{J} ) = 0$. Further, for all $M\in \modA$ there exists a short exact sequence $    0 \rightarrow tM \rightarrow M \rightarrow M/tM \rightarrow 0$ with $tM\in \ob \mathcal{I}$ by (a) and $M/tM \in \ob \mathcal{J}$ by (b). Hence,  $(\ob \mathcal{I}, \ob \mathcal{J})$ is a torsion pair.
\end{proof}
\end{prop}

We continue by discussing when epi-closed classes and mono-closed classes are functorially finite. For $M\in \modA$ let $\add M$ be the smallest full additive subcategory of $\modA$ containing $M$, let $\gen M$ be the closure of $\add M$ under factor modules and let $\cogen M$ be the closure of $\add M $ under submodules. Clearly $\gen M$ is the smallest epi-closed class and $\cogen M$ the smallest mono-closed class containing $M$. We reprove the classification of functorially finite epi-closed (mono-closed) classes in \cite[Proposition 4.6, Proposition 4.7]{AuslanderSmalo}, using the theory of ideal torsion pairs:

\begin{prop}\label{fff} \begin{itemize}
    \item[\rm (a)] A full additive subcategory $\mathcal{C}$ of $\modA$ is a functorially finite epi-closed class if and only if $\mathcal{C} = \gen M$ for some $M\in \modA$.
    \item[\rm (b)] A full additive subcategory $\mathcal{C}$ of $\modA$ is a functorially finite mono-closed class if and only if $\mathcal{C} = \cogen M$ for some $M\in \modA$.
\end{itemize}
\end{prop}

\begin{proof} (a) Let $\mathcal{C}$ be a functorially finite epi-closed class and $\varphi \colon A \rightarrow C_A$ a left $\mathcal{C}$-approximation. Then $C_A \in \mathcal{C}$ and for $M\in \mathcal{C}$ an epimorphism $A^n \rightarrow M$ must factor through $\varphi^n$ for some $n\in \mathbbm{N}$. We conclude that $\gen C_A = \mathcal{C}$.

For the other implication, consider the ideal $\mathcal{I} = \langle \add M \rangle$. If $\varphi_1, \dots, \varphi_n$ denotes a basis of the $k$-module $\Hom(A,M)$, then the induced morphism $\varphi \colon A \rightarrow M^n$ is a left $\mathcal{I}$-approximation. Thus, by Corollary \ref{coro} the ideal ${\rm I}(\mathcal{I})$ is functorially finite. By Lemma 4.3 the ideal ${\rm I}(\mathcal{I})$ equals $\langle \gen M \rangle$, as $\gen M$ is the smallest epi-closed class containing $M$. Hence,  $\gen M$ is functorially finite.

(b) Similar to (a).
\end{proof}

It would be nice to have a similar result as above, classifying all functorially finite torsion ideals and torsion-free ideals (not only those generated by epi-closed classes and mono-closed classes). To do so, we consider the abelian category $\morA$, where the objects are morphisms in $\modA$ and the morphisms are commutative squares of morphisms in $\modA$. Notice that $\morA$ is equivalent to the category of finitely generated (left) $B$-modules over the Artin algebra 
\begin{align*}
    B= \begin{bmatrix}
    A &0 \\
    A & A
\end{bmatrix}.
\end{align*}

\begin{lem}\label{coronot}\begin{itemize}
    \item[\rm (a)] If $\mathcal{I}$ is a torsion ideal of $\modA$, then $\mathcal{I}$ is an epi-closed class in $\morA$. If $\mathcal{C}$ is an epi-closed class in $\morA$, then the collection of all morphisms factoring through a morphism in $\mathcal{C}$ is a torsion ideal. 
    \item[\rm (b)] If $\mathcal{J}$ is a torsion-free ideal of $\modA$, then $\mathcal{J}$ is a mono-closed class in $\morA$. If $\mathcal{C}$ is a mono-closed class in $\morA$, then the collection of all morphisms factoring through a morphism in $\mathcal{C}$ is a torsion-free ideal. 
\end{itemize}
 \end{lem}

\begin{proof} (a) Let $\varphi \colon M \rightarrow N$ be in $\mathcal{I}$. An epimorphism in $\morA$ starting in $\varphi$ is given by a commutative square of morphisms
\begin{equation*}
    \begin{tikzcd}
        M \arrow[d, "f", swap] \arrow[r, "\varphi"] & N \arrow[d,"g"]\\
        M' \arrow[r, "\varphi'"] & N'
    \end{tikzcd}
\end{equation*}
with $f$ and $g$ surjective. Then $\varphi' f = g \varphi \in \mathcal{I}$. Hence,  $\varphi' \in \mathcal{I}$ by Lemma \ref{charai}. It follows that $\mathcal{I}$ is an epi-closed class in $\morA$.

Let $\mathcal{C}$ be an epi-closed class in $\morA$ and $\mathcal{I}$ the collection of all morphisms factoring through a morphism in $\mathcal{C}$. Clearly $\mathcal{I}$ is an ideal. To see that $\mathcal{I}$ is a torsion ideal, we make use of Lemma \ref{charai}. Let $\varphi \colon L \rightarrow N$ be a morphism and $f\colon M \rightarrow L$ an epimorphism with $\varphi f\in \mathcal{I}$. Then $\varphi f $ factors as $\varphi_1 \alpha \varphi_2$ for some $\alpha\colon M' \rightarrow N'$ in $\mathcal{C}$, $\varphi_1\colon N' \rightarrow N$ and $\varphi_2 \colon M \rightarrow M'$. Now consider the commutative square 
\begin{equation*}
\begin{tikzcd}
    M' \arrow[r, "\alpha"] \arrow[d, "\pi_1", swap] & N'\arrow[d, "\pi_2"]\\
    \im \varphi_1 \alpha \arrow[r, "\iota"] & \im \varphi_1
\end{tikzcd}
\end{equation*}
where $\pi_1 \colon  M'\rightarrow \im \varphi_1\alpha, \pi_2 \colon N'\rightarrow \im \varphi_1$ are the canonical projections and $\iota \colon \im \varphi_1\alpha \rightarrow \im \varphi_1$ the canonical inclusion. Because $\mathcal{C}$ is epi-closed and $\alpha \in \mathcal{C}$, we conclude that $\iota \in \mathcal{C}$. Now $\im \varphi = \im \varphi f = \im \varphi_1 \alpha \varphi_2 \subseteq \im \varphi_1 \alpha$. Hence,  $\varphi$ factors through $\iota \in \mathcal{C}$ and so $\varphi \in \mathcal{I}$. Thus,   $\mathcal{I}$ is a torsion ideal.

(b) Similar to (a).
\end{proof}

\begin{prop}\begin{itemize}\label{gen}\item[\rm (a)] An ideal $\mathcal{I}$ is a functorially finite torsion ideal if and only if there exists $\varphi \in \morA$ such that $\mathcal{I}$ equals the collection of all morphisms factoring through a morphism in $\gen \varphi$.
    \item[\rm (b)] An ideal $\mathcal{J}$ is a functorially finite torsion-free ideal if and only if there exists $\psi \in \morA$ such that $\mathcal{J}$ equals the collection of all morphisms factoring through a morphism in $\cogen \psi$.
\end{itemize}
\end{prop}

\begin{proof} (a) Let $\mathcal{I}$ be a functorially finite torsion ideal and $\varphi \colon A \rightarrow C_A$ a left $\mathcal{I}$-approximation. We show that $\mathcal{I}$ equals the collection of all morphisms factoring through a morphism in $\gen \varphi$. First $\gen \varphi \subseteq \mathcal{I}$ by Lemma \ref{coronot}. Let $f\colon M \rightarrow N$ be in $\mathcal{I}$ and $\pi \colon A^n \rightarrow M$ an epimorphism for some $n\in \mathbbm{N}$. Then $f \pi $ factors through $\varphi^n$. Hence, there exists $g \colon C_A^n \rightarrow N$ with $f \pi = g \varphi^n$. It follows that $\im f = \im f\pi = \im g \varphi^n \subseteq \im g$. Thus, there exists a commutative diagram of morphisms
\begin{equation*}
    \begin{tikzcd}
        A^n \arrow[d, "\pi", swap] \arrow[r, "\varphi^n"] &C_A^n \arrow[d, "\pi'"]\\
        M \arrow[r,"f_1"] &\im g \arrow[r, "f_2"] & N
    \end{tikzcd}
\end{equation*}
where $\pi'\colon C_A^n \rightarrow \im g$ denotes the canonical projection and $f = f_2 f_1$. Clearly $f_1 \in \gen \varphi$. It follows that $f$ factors through a morphism in $\gen \varphi$ and $\mathcal{I} = \gen \varphi$.

Let $\mathcal{I}$ equal the collection of all morphisms factoring through a morphism in $\mathcal{C} = \gen \varphi$ for some $\varphi \in \morA$. Then $\mathcal{I}$ is a torsion ideal by Lemma \ref{coronot}. Further, $\mathcal{C}$ is functorially finite in $\morA$ by Proposition \ref{fff}. Hence, for $M\in \modA$ there exists a left $\mathcal{C}$-approximation of $1_M$, given by a commutative square of morphisms
\begin{equation*}
\begin{tikzcd}
    M \arrow[d, "f_1", swap] \arrow[r, equals] &M\arrow[d, "f"] \\
    C_A \arrow[r, "f_2"] & D_A.
\end{tikzcd}
\end{equation*}
We show that $f$ is a left $\mathcal{I}$-approximation. First $f\in \mathcal{I}$ since $f_2 \in \mathcal{C}$ and $f = f_2f_1$. Now let $g\colon M \rightarrow N$ be a morphism factoring through a morphism $\alpha \colon M' \rightarrow N'$ in $\mathcal{C}$. Then $g = g_2 \alpha g_1$ for $g_1 \colon M \rightarrow M'$ and $g_2 \colon N' \rightarrow N$. We consider the commutative square
\begin{equation*}
\begin{tikzcd}
    M \arrow[d, "g_1", swap] \arrow[r, equals] &M\arrow[d, "\alpha g_1"] \\
    M' \arrow[r, "\alpha"] & N'.
\end{tikzcd}
\end{equation*}
It must factor through the left $\mathcal{C}$-approximation of $1_M$ from above, so there exists a commutative diagram of morphisms
\begin{equation*}
\begin{tikzcd}
    M \arrow[d, "f_1", swap] \arrow[r, equals] &M\arrow[d, "f"] \\
    C_A \arrow[r, "f_2"] \arrow[d, "h_1", swap] & D_A \arrow[d, "h_2"]\\
    M' \arrow[r, "\alpha"] & N'
\end{tikzcd}
\end{equation*}
with $g_1 = h_1 f_1$. In total $g = g_2 \alpha g_1 = g_2 \alpha h_1 f_1 = g_2 h_2 f$. It follows that $f$ is a left $\mathcal{I}$-approximation, so $\mathcal{I}$ is functorially finite.

(b) Similar to (a)
\end{proof}
\section{Transfinite powers of the radical and related ideal torsion pairs}

Preprojective modules were first introduced by Dlab and Ringel \cite{Dlab} for finite dimensional tensor algebras and later generalized by Auslander and Smal{\o} for arbitrary Artin algebras \cite{AuslanderSmalo}. Our aim is to extend the class of preprojective modules using the theory of ideal torsion pairs. Let us mention that Krause also extended the class of preprojective modules in \cite{Krause}, using a different approach. In what follows, all results can be dualized leading to an extension of the class of preinjective modules.

Following Prest \cite{Prest}, we define the notion of transfinite powers of an ideal $\mathcal{I}$ of $\modA$: For $n \in \mathbbm{N}$ let $\mathcal{I}^n$ denote the collection of all $n$-fold compositions of morphisms in $\mathcal{I}$ (in particular $\mathcal{I}^0 = \Hom$). If $\lambda$ is a non-zero limit ordinal, let $\mathcal{I}^\lambda = \bigcap_{\alpha < \lambda} \mathcal{I}^\lambda$. If $\alpha$ is an arbitrary infinite ordinal, then $\alpha = \lambda + n$ for a limit ordinal $\lambda$ and $n\in \mathbb{N}$, and we let $\mathcal{I}^\alpha = (\mathcal{I}^\lambda)^{n+1}$. Lastly, we define $\mathcal{I}^\infty = \bigcap_{\alpha} \mathcal{I}^\alpha$. The case $\mathcal{I} = \radA$ will be most important to us.

\begin{rem}\label{local}\rm From the definition of $\mathcal{I}^\alpha$, it is not hard to see that $\mathcal{I}^\alpha$ is an ideal for all ordinal numbers $\alpha$. This yields a descending chain of ideals 
\begin{align*}
    \Hom \supseteq \mathcal{I} \supseteq \mathcal{I}^2 \supseteq \dots \supseteq \mathcal{I}^\alpha \supseteq \mathcal{I}^{\alpha+1} \supseteq \dots
\end{align*}
For $M,N\in \modA$ the expression $\mathcal{I}^\alpha (M,N)$ is a submodule of the finitely generated $k$-module $\Hom (M,N)$. Thus,   if $\lambda$ is a limit ordinal (or $\infty$), then $\mathcal{I}^\lambda (M,N) = \bigcap_{\alpha < \lambda} \mathcal{I}^\alpha (M,N) = \mathcal{I}^\alpha (M,N)$ for some $\alpha < \lambda$.
\end{rem}


The \emph{projective rank} of a module $M \in \modA$, denoted by $\prk M$, is the smallest ordinal number $\alpha$ (or $\infty$) with $M\in \ob \rm{I}(\radA^\alpha)$. Now $M \in \ob \rm{I}(\radA^\alpha)$ is equivalent to $\varphi \in \radA^\alpha$ for all $\varphi \colon A\rightarrow M$ by Remark \ref{smallt}. Hence, $\prk M$ equals the smallest $\alpha$ such that there exists $\varphi \colon A \rightarrow M$ in $\radA^\alpha \backslash \radA^{\alpha+1}$ (or $\infty$ if no such $\alpha$ exists). In particular if $M$ is indecomposable, then $\prk M = 0$ is equivalent to $M$ being projective, since exactly in that case there exists a split epimorphism $A\rightarrow M$.

\begin{rem}\label{calc} \rm Let $\alpha$ be an ordinal number (or $\infty$). Then $\ob \rm{I}(\radA^\alpha)$ is an epi-closed class by Lemma \ref{prooop}. This has the following immediate consequences:
\begin{itemize}
    \item[(i)] For every epimorphism $M \rightarrow N$ in $\modA$ we have $\prk M \leq \prk N$. 
    \item[(ii)] For $M, N \in \modA$ we have $\prk M \oplus N = \min \{\prk M, \prk N\}$.
\end{itemize}
\end{rem}

Recall that an indecomposable module $X\in \modA$ is \emph{preprojective} if there exists a finite collection $\mathcal{C}$ of indecomposable modules with the following property: If there exists an epimorphism $M\rightarrow X$, then $M$ must have a direct summand isomorphic to a module in $\mathcal{C}$ (see \cite{AuslanderSmalo}). In what follows, we show that the indecomposable preprojective modules are precisely those of finite projective rank.

\begin{lem}\label{lem} The ideal $\radA^n$ is functorially finite for all $n \in \mathbbm{N}$. Moreover, if $\varphi \colon A\rightarrow C_A$ is a left minimal $\radA^n$-approximation, then an indecomposable module $X\in \modA$ is isomorphic to a direct summand of $C_A$ if and only if there exists $\psi \colon A \rightarrow X$ in $\radA^{n}\backslash\radA^{n+1}$.
\end{lem}

\begin{proof} By Proposition \ref{radical} the ideal $\radA$ is functorially finite. For $M\in \modA$ consider a sequence
\begin{align*}
    M = M_0 \xrightarrow{f_1} M_1 \xrightarrow{f_2} \dots \xrightarrow{f_n} M_{n}
\end{align*}
of left $\radA$-approximations $M_i \rightarrow M_{i+1}$. We show that $f = f_n f_{n-1} \dots f_1$ is a left $\radA^n$-approximation. Clearly $f\in \radA^n$. Let $g\in \radA^n$ starting in $M$. Then $g$ factors as $g_n g_{n-1} \dots g_1$ with $g_i\colon N_{i-1} \rightarrow N_i$ in $\radA$ and $N_0 = M$. Because $f_i$ it a left $\radA$-approximation for all $i$, we obtain a commutative diagram
\begin{equation*}
    \begin{tikzcd}
        M \arrow[r, "f_1"] \arrow[d, "g_1", swap] & M_1 \arrow[r, "f_2"] \arrow[dl] & \dots  \arrow[r, "f_m"] & M_m \arrow[dddlll]\\
        N_1 \arrow[d, "g_2",swap] & \ddots \\
        \vdots \arrow[d, "g_m", swap] \\
        N_m.
    \end{tikzcd}
\end{equation*}
It follows that $g$ factors through $f$. Hence, $f$ is a left $\radA^n$-approximation and the ideal $\radA^n$ is covariantly finite. Similarly $\radA^n$ is contravariantly finite and thus functorially finite.

Let $X\in \modA$ be indecomposable such that there exists $\psi \colon A\rightarrow X$ in $\radA^n \backslash \radA^{n+1}$. Then $\psi$ factors through the left $\radA^n$-approximation $\varphi \colon A \rightarrow C_A$, so $\psi = g \varphi$ for $g \colon C_A \rightarrow X$. Further, the morphism $g$ can not be radical, as otherwise $\psi = g \varphi \in \radA^{n+1}$. It follows that $g$ is a split epimorphism, since $X$ is indecomposable. Hence, $X$ is a direct summand of $C_A$.

Let $C_A = X\oplus Y$ with $X$ indecomposable, let $\pi_X \colon C_A \rightarrow X, \pi_Y \colon C_A \rightarrow Y$ be the canonical projections and $\iota_X \colon X \rightarrow C_A, \iota_Y \colon Y \rightarrow C_A$ the canonical inclusions. For a contradiction suppose that $\radA^n(A,X) = \radA^{n+1}(A,X)$. Then $\pi_X \varphi \in \radA^{n+1}$ factors as $\pi_X \varphi = g f$ for $f \in \radA^n$ and $g \in \radA$. Now $f$ factors through the left minimal $\radA^n$-approximation $\varphi$, so $f = h \varphi$. Hence,
\begin{align*}
    \varphi = (\iota_X \pi_X + \iota_Y \pi_Y)\varphi = (\iota_X gh+\iota_Y \pi_Y) \varphi.
\end{align*}
Because $\varphi$ is left minimal, the morphism $\iota_X gh+\iota_Y \pi_Y$ is an isomorphism. Thus, there exists $\psi \colon C_A \rightarrow C_A$ with $\psi(\iota_X gh+\iota_Y \pi Y) = 1_{C_A}$. We conclude that
\begin{align*}
    1_X = \pi_X 1_{C_A} \iota_X = \pi _X \psi(\iota_X gh+\iota_Y \pi_Y) \iota_X = \pi _X \psi \iota_X gh \iota_X.
\end{align*}
Now $g \in \radA$ implies $1_X \in \radA$, which is a contradiction.
\end{proof}

\begin{coro}\label{cororo} Let $X\in \modA$ be indecomposable and $A \rightarrow C_{n}$ a left minimal $\radA^n$-approximation for all $n\in \mathbbm{N}$. Then $X$ can only be isomorphic to a direct summand of $C_{n}$ for finitely many $n$.
\end{coro}

\begin{proof} Suppose that $X$ is a direct summand of infinitely many $C_{n}$. Then the sequence $\radA(A,X) \supseteq \radA^2(A,X) \supseteq \dots$ of submodules of the $k$-module $\Hom(A,X)$ would have infinitely many proper inclusions by Lemma \ref{lem}, which contradicts that $\Hom(A,X)$ is of finite length as a $k$-module.
\end{proof}

\begin{prop}\label{preproj} Let $X\in \modA$ be indecomposable. The projective rank of $X$ is finite if and only if $X$ is preprojective.
\end{prop}

\begin{proof} Let $\prk X = n \in \mathbbm{N}$ and $\mathcal{C}$ be the collection of all indecomposable modules $Y \in \modA$ with $\prk Y \leq n$. Then for all $Y\in \mathcal{C}$ there exists $A \rightarrow Y$ in $\radA^{m}\backslash \radA^{m+1}$ for some $m\leq n$. Thus, by Lemma \ref{lem}, the number of isomorphism types of modules in $\mathcal{C}$ is finite. Now if there exists an epimorphism $M\rightarrow X$ such that $M$ does not have a direct summand isomorphic to a module in $\mathcal{C}$, then
\begin{align*}
    \prk X = n \geq \prk M > n
\end{align*}
by Remark \ref{calc}. This is a contradiction. It follows that $X$ is preprojective.

Let $X$ be preprojective and $\mathcal{C}$ a finite collection of indecomposable modules such that for every epimorphism $M\rightarrow X$ the module $M$ has a direct summand isomorphic to a module in $\mathcal{C}$. By Corollary \ref{cororo} there exists $n\in \mathbbm{N}$ such that for a left $\radA^{n}$-approximation $A\rightarrow C_A$ the module $C_A$ does not contain a direct summand isomorphic to a module in $\mathcal{C}$. For a contradiction suppose that $X$ has non-finite projective rank. Then an epimorphism $A^m\rightarrow X$ is contained in $\radA^n$ so it must factor through $A^m \rightarrow C_A^m$. This yields an epimorphism $C_A^m \rightarrow X$, which is a contradiction since $C_A^m$ does not have a direct summand isomorphic to a module in $\mathcal{C}$.
\end{proof}

We continue by investigating the modules of non-finite projective rank. It is not clear for which ordinal number $\alpha$ (or $\infty$) there exists a non-zero module $X\in \modA$ of projective rank $\alpha$. In what follows, we show a criteria for this.

\begin{lem}\label{inf} Let $\mathcal{I}$ be an ideal of $\modA$ and $\lambda \neq 0$ a limit ordinal (or $\infty$). Then ${\rm I}(\mathcal{I}^\lambda) = \langle \ob {\rm I}(\mathcal{I}^\lambda) \rangle$. In particular $\ob {\rm I}(\mathcal{I}^\lambda)  \neq 0$ if and only if $\mathcal{I}^\lambda \neq 0$.
\end{lem}

\begin{proof} Clearly $\langle \ob {\rm I}(\mathcal{I}^\lambda)\rangle \subseteq {\rm I}(\mathcal{I}^\lambda)$. Let $M \in \modA$, $t$ the subfunctor of $\mathbbm{1}_A$ corresponding to ${\rm I}(\mathcal{I}^\lambda)$, $\iota\colon tM \rightarrow M$ the canonical inclusion and $\pi \colon A^n \rightarrow tM$ an epimorphism. Then $\iota \pi\in \mathcal{I}^\lambda$ by Remark \ref{smallt}. By Remark \ref{local} we can choose $\alpha < \lambda$ with $\mathcal{I}^\lambda (A, M) = \mathcal{I}^\alpha (A,M)$ and $\mathcal{I}^\lambda (A,tM) = \mathcal{I}^\alpha (A,tM)$. Further, $\iota \pi$ factors as $g f$ for $f, g \in \mathcal{I}^\alpha$ since $(\mathcal{I}^\alpha)^2 \subseteq \mathcal{I}^\lambda$. Now for all $\varphi$ starting in $A$ the composition $g \varphi$ is in $\mathcal{I}^\alpha(A,M) = \mathcal{I}^\lambda (A,M)$. Hence, $g \in {\rm I}(\mathcal{I}^\lambda)$ factors through $\iota$ and $g = \iota \psi$. Since $ \psi f \in \mathcal{I}^\alpha (A^n, tM) = \mathcal{I}^\lambda (A^n, tM)$, we deduce that $\psi f$ factors through the canonical inclusion $\iota' \colon t(tM) \hookrightarrow tM$. Thus,   $\iota = gf = \iota \psi f$ factors through $\iota'$ and we conclude that $t(tM) = tM$. It follows that $tM \in \ob {\rm I}(\mathcal{I}^\lambda)$ and hence ${\rm I}(\mathcal{I}) = \langle \ob {\rm I}(\mathcal{I}^\lambda) \rangle$. 

Further, if $\mathcal{I}^\lambda \neq 0$ then ${\rm I}(\mathcal{I}^\lambda) \neq 0$ and thus $\ob {\rm I}(\mathcal{I}^\lambda) \neq 0$. On the contrary, if $\ob {\rm I}(\mathcal{I}^\lambda) \neq 0$ then ${\rm I}(\mathcal{I}^\lambda) \neq 0$. In that case $\mathcal{I}^\lambda (A,-) \neq 0$ by Remark \ref{smallt}.
\end{proof}

\begin{coro}\label{exi} Let $\lambda\neq 0$ be a limit ordinal (or $\infty$). There exists a non-zero module of projective rank greater than or equal to $\lambda$ if and only if $\radA^\lambda \neq 0$.
\end{coro}




For the finite case we have already seen that for each $n\in \mathbbm{N}$ there can only be finitely many isomorphism types of indecomposable modules of projective rank $n$. The infinite case is completely opposite. We will show that if there exists an indecomposable module of projective rank $\alpha$ for an infinite ordinal number $\alpha$ (or $\alpha = \infty$), then there exist infinitely many isomorphism types of indecomposable modules of projective rank close to $\alpha$. 

\begin{lem}\label{need} Let $\alpha$ be a non-finite ordinal number (or $\infty)$ and write $\alpha = \lambda + n$ for a limit ordinal $\lambda$ (or $\infty$) and $n\in \mathbbm{N}$. For all $M\in \modA$ with $\prk M = \alpha$ there exists a radical epimorphism $M' \rightarrow M$ with $\lambda \leq \prk M' \leq \alpha$. 
\end{lem}

\begin{proof} Let $t$ be the subfunctor of $\mathbbm{1}_A$ corresponding to $\rm{I}(\radA^\lambda)$, let $\varphi \colon C_M \rightarrow M$ be a right $\radA$-approximation and $\pi \colon A^n \rightarrow M$ an epimorphism. The morphism $\pi$ is contained in $\radA^\alpha \subseteq \radA^\lambda$ since $\prk M = \alpha$.  We can choose $\beta < \lambda$ with $\mathcal{I}^\lambda (A, C_M) = \mathcal{I}^\beta (A,C_M)$ by Remark \ref{local}. Further, the morphism $\pi$ factors as $g f$ for $f, g \in \radA^\beta$ since $(\radA ^\beta)^2 \subseteq \radA^\lambda$. In particular $g$ must factor through the right $\radA$-approximation $\varphi$, so $g = \varphi g'$. Now $g'f \in \mathcal{I}^\beta (A,C_M) = \mathcal{I}^\lambda (A,C_M)$. Hence, $g'f$ factors through the canonical inclusion $\iota\colon t C_M \rightarrow C_M$, so $g'f = \iota f'$. This yields $\pi = gf = \varphi g' f = \varphi \iota f'$. Because $\pi$ is an epimorphism, it follows that $\varphi \iota$ is an epimorphism, which is also radical since $\varphi \in \radA$. The morphism $\varphi \iota$ starts in $tC_M$, which is contained in $\ob \rm{I}(\mathcal{I}^\lambda)$ by Lemma \ref{inf}. We conclude that $\prk tC_M \geq \lambda$. By Remark \ref{calc} the inequality $\prk tC_M \leq \prk M = \alpha$ holds. This proves the claim with $M' = tC_M$.
\end{proof}

\begin{prop}\label{chain} Let $X\in \modA$ be indecomposable of non-finite projective rank and write $\prk X  = \lambda + n$ for a limit ordinal $\lambda$ (or $\infty$) and $n\in \mathbbm{N}$. There exists a chain
\begin{align*}
    X = X_0 \xleftarrow{f_1} X_1 \xleftarrow{f_2} X_2 \xleftarrow{f_3} \dots
\end{align*}
of radical morphisms $f_i$ between indecomposable modules $X_i$ with $f_1 f_2 \dots f_i \neq 0$ and $\lambda \leq \prk X_i \leq \alpha$ for all $i$.
\end{prop}

\begin{proof} We construct modules $M_i, N_i$ recursively such that every indecomposable direct summand of $M_i$ has projective rank in the interval $[\lambda, \alpha]$ and $\prk N_i > \alpha$. Further, we construct epimorphisms
\begin{align*}
\varphi_i \colon M_i \oplus N_i \xrightarrow{\begin{pmatrix} \alpha_i & \beta_i\\0& \gamma_i \end{pmatrix}} M_{i-1}\oplus N_{i-1}
\end{align*}
with $\alpha_i\in \radA$. Let $M_0 = X$ and $N_0 = 0$. Suppose given $M_{i-1}, N_{i-1}$. By Lemma \ref{need}  there exists an epimorphism $\varphi \colon M' \rightarrow M_{i-1}$ with $\prk M' \in [\lambda, \alpha]$ and $\varphi \in \radA$. By Remark \ref{calc} (ii) we can decompose $M' = M \oplus N$ such that every indecomposable direct summand of $M$ has projective rank in $[\lambda,\alpha]$ and $\prk N> \alpha$. We define $M_{i} = M$ and $N_{i} = N_{i-1} \oplus N$. The morphism $\varphi \oplus 1_{N_{i-1}}$ induces the desired epimorphism $\varphi_{i}$ (via $M_i \oplus N_i = M' \oplus N_{i-1}$).

The composition $\varphi_1 \varphi_2 \dots \varphi_i$ is an epimorphism for all $i$. Further, $\varphi_1 \varphi_2 \dots \varphi_i$ equals 
\begin{align*}
    \begin{pmatrix} \alpha_1 & \beta_1\\0& \gamma_1 \end{pmatrix} \begin{pmatrix} \alpha_2 & \beta_2\\0& \gamma_2 \end{pmatrix} \dots \begin{pmatrix} \alpha_i & \beta_i\\ 0& \gamma_i \end{pmatrix} = \begin{pmatrix} \alpha_1 \alpha_2 \dots \alpha_i & \beta\\0 & \gamma_1 \gamma_2 \dots \gamma_i \end{pmatrix}
\end{align*}
for some $\beta$. If $\alpha_1 \alpha_2 \dots \alpha_i = 0$ for some $i$, then $\beta$ and $\gamma_1\gamma_2 \dots \gamma_i$ would induce an epimorphism $N_i \rightarrow M_0 \oplus N_0 = X$. By Remark \ref{calc} (i) this yields
\begin{align*}
    \alpha < \prk N_i \leq \prk X \leq \alpha,
\end{align*}
which is a contradiction. Hence, $\alpha_1 \alpha_2 \dots \alpha_i \neq 0$ for all $i$. Let $M_i = \bigoplus_{j=1}^{n_i} X_{ij}$ with $X_{ij}$ indecomposable. Then by construction $\prk X_{ij} \in [\lambda, \alpha]$ and $\alpha_i$ induces radical morphisms $\alpha_{ij}^{j'} \colon X_{ij} \rightarrow X_{(i-1)j'}$. We have
\begin{align*}
    0 \neq \alpha_{1} \alpha_{2} \dots \alpha_{i} = \sum_{j_0, j_1, \dots, j_i} \alpha_{1 j_1}^{j_0} \alpha_{2 j_2}^{j_1} \dots \alpha_{i j_i}^{j_{i-1}}.
\end{align*}
Thus, for all $i\in \mathbbm{N}$ there exist $j_0, j_1, \dots , j_i$ with $\alpha_{1 j_1}^{j_0} \alpha_{2 j_2}^{j_1} \dots \alpha_{i j_i}^{j_{i-1}} \neq 0$. As in \cite[p. 474]{Bass}, the K\"onig Graph Theorem now implies the existence of an infinite sequence $j_0, j_1, \dots$ such that $\alpha_{1 j_1}^{j_0} \alpha_{2 j_2}^{j_1} \dots \alpha_{i j_i}^{j_{i-1}} \neq 0$ for all $i$. 
\end{proof}

\begin{coro}\label{unbounded} Let $\lambda$ be a non-zero limit ordinal. For all $n\in \mathbbm{N}$ either no indecomposable modules have projective rank in $[\lambda, \lambda+n]$ or the length of indecomposable modules having projective rank in $[\lambda, \lambda+n]$ is unbounded.
\end{coro}

\begin{proof} Combine Proposition \ref{chain} and the Harada-Sai lemma. 
\end{proof}

\begin{exa}\label{kro}\rm Let $k$ be an algebraically closed field, $Q = \begin{tikzcd}
    1 \arrow[r, shift left] \arrow[r, shift right] & 2
\end{tikzcd}$ the Kronecker quiver and $A = kQ$. We can divide the indecomposable modules of $\modA$ into three parts: The preprojective modules $\mathcal{P}$, the regular modules  $\mathcal{R}$ and the preinjective modules $\mathcal{I}$. The Auslander-Reiten quiver of $\modA$ can be visualized as follows.
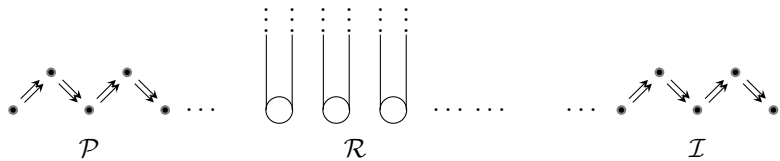
\tikzstyle{place}=[circle,draw=black!50,fill=black!100,thick,
inner sep=0pt,minimum size=1mm]
\begin{align*}
\begin{tikzpicture}
\draw (1, -0.5) node{$\mathcal{P}$};
\draw (9, -0.5) node{$\mathcal{I}$};
\draw (4.5, -0.5) node{$\mathcal{R}$};
\draw (0,0) node[place]{};
\draw [-stealth](0.1,0.15) -- (0.35,0.4);
\draw [-stealth](0.15,0.1) -- (0.4,0.35);
\draw (0.5,0.5) node[place]{};
\draw [-stealth](0.65,0.4) -- (0.9,0.15);
\draw [-stealth](0.6,0.35) -- (0.85,0.1);
\draw (1,0) node[place]{};
\draw [-stealth](1.1,0.15) -- (1.35,0.4);
\draw [-stealth](1.15,0.1) -- (1.4,0.35);
\draw (1.5,0.5) node[place]{};
\draw [-stealth](1.65,0.4) -- (1.9,0.15);
\draw [-stealth](1.6,0.35) -- (1.85,0.1);
\draw (2,0) node[place]{};
\draw (2.5,0) node{$\dots$};
\draw (10,0) node[place]{};
\draw [-stealth](9.1,0.15) -- (9.35,0.4);
\draw [-stealth](9.15,0.1) -- (9.4,0.35);
\draw (9.5,0.5) node[place]{};
\draw [-stealth](9.65,0.4) -- (9.9,0.15);
\draw [-stealth](9.6,0.35) -- (9.85,0.1);
\draw (9,0) node[place]{};
\draw [-stealth](8.65,0.4) -- (8.9,0.15);
\draw [-stealth](8.6,0.35) -- (8.85,0.1);
\draw (8.5,0.5) node[place]{};
\draw [-stealth](8.1,0.15) -- (8.35,0.4);
\draw [-stealth](8.15,0.1) -- (8.4,0.35);
\draw (8,0) node[place]{};
\draw (7.5,0) node{$\dots$};
\draw (3.5,0) circle (5pt);
\draw (3.33,0)--(3.33,1);
\draw (3.33,1.3) node{$\vdots$};
\draw (3.67,0)--(3.67,1);
\draw (3.67,1.3) node{$\vdots$};
\draw (4.25,0) circle (5pt);
\draw (4.08,0)--(4.08,1);
\draw (4.08,1.3) node{$\vdots$};
\draw (4.42,0)--(4.42,1);
\draw (4.42,1.3) node{$\vdots$};
\draw (5,0) circle (5pt);
\draw (4.83,0)--(4.83,1);
\draw (4.83,1.3) node{$\vdots$};
\draw (5.17,0)--(5.17,1);
\draw (5.17,1.3) node{$\vdots$};
\draw (5.75,0) node{$\dots$};
\draw (6.3,0) node{$\dots$};
\end{tikzpicture}
\end{align*}
\end{exa}
By Proposition \ref{preproj} it follows that the modules in $\mathcal{P}$ are exactly those of finite projective rank. Every morphism from $\mathcal{P}$ to $\mathcal{R}$ is in $\radA^\omega$ where $\omega$ denotes the smallest non-finte ordinal number. Further, every morphism from $\mathcal{P}$ to $\mathcal{I}$ is in $(\radA^\omega)^2 = \radA^{\omega+1}$. Since $A\in \mathcal{P}$, it follows that $\mathcal{R}$ contains exactly the modules of projective rank $\omega$ and $\mathcal{I}$ those of projective rank $\omega +1$. By Corollary \ref{unbounded} we would expect the length of modules in $\mathcal{R}$ to be unbounded, which is the case.
\section{The torsion dimension and the Krull-Gabriel dimension}

Through the concept of ideal torsion pairs, we introduce a new homological dimension for Artin algebras and relate it to the Krull-Gabriel dimension. The \emph{torsion-dimension} of $A$, denoted by $\TD(A)$, is the m-dimension of the lattice of functorially finite ideal torsion pairs of $A$.

\begin{prop}\label{small} The inequality $\TD(A) \leq \KG(A)$ holds.
\end{prop}

\begin{proof} Let $U\colon \modA \rightarrow \Ab$ be the forgetful functor, $L$ the lattice of finitely presented subfunctors of $f$ and $L'$ the lattice of finitely presented subfunctors of $\mathbbm{1}_A$. Note that in the first case we consider additive functors from $\modA$ to $\Ab$, while in the second case we consider them to be from $\modA$ to $\modA$. Clearly $L' \rightarrow L, F \mapsto U F$ defines an injective lattice homomorphism. Hence, $\dim L' \leq \dim L$. Since $U \cong \Hom(A,-)$, the equality $\dim L = \KG(A)$ holds by Proposition \ref{mdim}. Further, we have $\dim L' = \TD(A)$ by Corollary \ref{corres}. It follows that $\TD(A) \leq \KG(A)$.
\end{proof}

\begin{rem}\label{commutative} \rm Suppose that $A$ is commutative and let $U\colon \modA \rightarrow \Ab$ be the forgetful functor. If $t$ is a subfunctor of $U$, then $tX$ is an abelian subgroup of $X$. Because multiplication with an element in $A$ is a morphism in $\modA$, it follows from the functoriality of $t$ that $tX$ is a submodule of $X$. Thus, subfunctors of $\mathbbm{1}_A$ coincide with subfunctors of $U$. Hence, $\KG(A) = \TD(A)$ in that case.   
\end{rem}

Another link between the torsion dimension and the Krull-Gabriel dimension of $A$ is given by the radical ideal $\radA$ of $\modA$ and its ordinal powers $\radA^\alpha$. First, we describe the relation between $\radA^\alpha$ and the Krull-Gabriel dimension. Schr\"oer conjectured that $\KG(A) =  n$ if and only if $\radA^{\omega (n-1)} \neq 0$ and $\radA^{\omega n} = 0$ for $n\in \mathbbm{N}_{>1}$, where $\omega$ denotes the first non-finite ordinal \cite{Schroer}. He defined the Krull-Gabriel dimension in such a way that it only takes values in $\mathbbm{N}\cup \{ \infty \}$. We state a version of his conjecture that also accounts for ordinal numbers. Note that there is no known example of an Artin algebra $A$ with $\omega \leq \KG(A) < \infty$. 

\begin{conj}\label{schr} Let $\alpha$ be a non-zero ordinal number. 
\begin{itemize}
    \item[\rm (i)] If $\KG(A) \geq \alpha+1$, then $\radA^{\omega \alpha} \neq 0$.
    \item[\rm (ii)] If $\radA^{\omega \alpha} \neq 0$, then $\KG(A) \geq \alpha+1$.
\end{itemize}
\end{conj}

Clearly the above conjecture is equivalent to: $\KG(A) = \alpha +1$ if and only if $\radA^{\omega \alpha} \neq 0$ and $\radA^{\omega(\alpha +1)} = 0$. This is exactly Schr\"oers conjecture for $\alpha \in \mathbbm{N}_{>0}$. Further, (ii) is proven for $\alpha = 1$ \cite{Herzog} and Krause showed that if $\radA^{\omega \alpha} \neq 0$, then $\KG(A) \geq \alpha$ \cite[Cor. 8.14]{Krause}.
The whole conjecture is proven if $A$ is a string algebra over an algebraically closed field $k$ \cite[Corollary 1.3]{Laking}.

The next aim will be to connect the torsion dimension and ordinal powers of the radical ideal. To do so, we give a different view on the lattice of finitely presented subfunctors of the forgetful functor $U \colon \modA \rightarrow \Ab$ and the lattice of finitely presented subfunctors of $\mathbbm{1}_A$.

\begin{rem}\rm Let $U\colon \modA \rightarrow \Ab$ be the forgetful functor. By the isomorphism $U\cong \Hom(A,-)$, we can identify subfunctors of $U$ with subfunctors of $\Hom(A,-)$. A finitely presented subfunctor of $\Hom(A,-)$ equals $\im \Hom(\varphi,-)$ for some $\varphi \colon A \rightarrow M$. Further, $\im \Hom(\varphi,-)$ is a subfunctor of $\im \Hom(\psi,-)$ if and only if $\varphi$ factors through $\psi$. Let $L$ be the collection of equivalence classes of morphisms starting in $A$ with the equivalence relation $\varphi \sim \psi$ if $\varphi$ factors through $\psi$ and $\psi$ factors through $\varphi$. Then $L$ is partially ordered by $\varphi \leq \psi$ if $\varphi$ factors through $\psi$. By the above observations, we can identify the lattice of finitely presented subfunctors of $U$ with $L$.
\end{rem}

\begin{lem}\label{dif} Let $U\colon \modA \rightarrow \Ab$ be the forgetful functor. The identification of finitely presented subfunctors of $U$ and equivalence classes of morphisms starting in $A$ restricts to an identification between finitely presented subfunctors of $\mathbbm{1}_A$ and equivalence classes of morphisms $\varphi\colon A\rightarrow M$ such that for all $a\in A$ there exists $\alpha\colon M\rightarrow M$ with $\varphi(a) = \alpha \varphi(1)$.
\end{lem}

\begin{proof} The isomorphism $\Hom(A,-) \cong U$ is given by $\psi \mapsto \psi(1)$ for $N\in \modA$ and $\psi \in \Hom(A,N)$. Hence, for $\varphi\colon A \rightarrow M$ the functor $\im \Hom(\varphi,-)$ becomes a subfunctor of $\mathbbm{1}_A$ under the isomorphism if and only if for $N\in \modA$ the abelian group
\begin{align*}
    N' = \{\psi(1) \mid \psi  \in \im\Hom(\varphi, N)\} = \{\alpha \varphi(1) \mid \alpha \colon M \rightarrow N \}
\end{align*}
is a submodule of $N$. In particular, in that case for all $a\in A$ there exists $\alpha\colon M \rightarrow M$ with $\varphi(a) = \alpha \varphi(1)$. Now this property already suffices, since for arbitrary $\beta \colon M \rightarrow N$ we have $a \beta \varphi(1) = \beta \varphi(a) = (\beta \alpha )\varphi(1)$ and so $N'$ is a submodule of $N$. 
\end{proof}

\begin{lem}\label{bigger} Let $A$ be generated by $a_1, a_2, \dots, a_n$ as a $k$-module. For a morphism $\varphi\colon A \rightarrow M$ let $\overline{\varphi}\colon A \rightarrow M^n$ be defined by $\overline{\varphi}(1) = (\varphi(a_1)\, \varphi(a_2)\,\dots \varphi(a_n))^\top$. Then for all $a\in A$ there exists $\alpha\colon M^n \rightarrow M^n$ with $\overline{\varphi}(a) = \alpha \overline{\varphi}(1)$.
\end{lem}

\begin{proof} For $a\in A$ let $a a_i = \sum_{i=1}^n c_{ij} a_j$ with $c_{ij} \in k$. Then the matrix $(c_{ij})_{i,j}$ defines a morphism $\alpha\colon M^n \rightarrow M^n$ such that the $i$'th component of $\alpha \overline{\varphi}(1)$ equals
\begin{align*}
    \sum_{j=1}^n c_{ij} \varphi(a_j) = \varphi\left(\sum_{j=1}^n c_{ij} a_j\right) = \varphi(a a_i),
\end{align*}
which is exactly the $i$'th component of $\overline{\varphi}(a)$. Hence, $\alpha \overline{\varphi}(1) = \overline{\varphi}(a)$.
\end{proof}

We are now ready to prove a connection between the torsion dimension of $A$ and ordinal powers of the radical ideal of $\modA$. It is similar to the result of Krause that $\radA^{\omega \alpha} \neq 0$ implies $\KG(A)\geq \alpha$. 

\begin{prop}\label{okay}Let $\alpha$ be a non-zero ordinal. If $\radA^{\omega \alpha} \neq 0$, then $\TD (A)\geq \alpha$.
\end{prop}

\begin{proof} Let $\psi \colon M \rightarrow N$ be a non-zero morphism in $\radA^{\omega \alpha}$ and $\varphi\colon A \rightarrow M$ with $\psi \varphi \neq 0$. Consider $n\in \mathbbm{N}$ and the morphisms $\overline{\varphi}$, $\overline{\psi \varphi}$ as in Lemma \ref{bigger}. By construction $\overline{\psi \varphi} = \psi^n \overline{\varphi}$. By Lemma \ref{dif} the morphisms correspond to finitely presented subfunctors $t'\leq t$ of $\mathbbm{1}_A$. We show that the m-dimension of the interval $[t',t]$ is greater than or equal to $\alpha$ by transfinite induction.

For $\alpha = 1$ we have $\psi \in \radA^\omega$ so for all $m$ there exists $\psi_i \colon M_i \rightarrow M_{i+1}$ in $\radA$ for $i= 1,2,\dots,m$ with $\psi = \psi_m \dots \psi_1$. For $i=0,1,\dots,m$ consider the finitely presented subfunctor $t_i$ of $\mathbbm{1}_A$ corresponding to $f_i = \overline{\psi_i\dots\psi_1 \varphi}$. Then $f_i = \psi_i^n f_{i-1}$ and thus $t' = t_m \leq t_{m-1} \leq \dots \leq t_0 = t$. Suppose $t_i = t_{i-1}$ for some $i$. Then $f_{i-1}$ would factor through $f_i$ and there exists $g\colon M_{i} \rightarrow M_{i-1}$ with $f_{i-1}= g f_i = g\psi_i^n f_{i-1} = \alpha f_{i-1}$ for $\alpha = g \psi_i^n \in \radA$. Hence, there exists $l>0$ with $\alpha^l = 0$ and $f_{i-1} = \alpha^l f_{i-1} = 0$. But then $\overline{\psi\varphi} = f_m = 0$ and thus $\psi \varphi = 0$, which is a contradiction. Hence, $t' = t_m \lneq t_{m-1} \lneq \dots \lneq t_0 = t$. Because $m$ was arbitrary, it follows that the interval $[t',t]$ has m-dimension greater than or equal to $1$.

Assume the claim holds for $\alpha>0$ and let $\psi \in \radA^{\omega (\alpha+1)} = \bigcap_{i=1}^\infty (\radA^{\omega \alpha})^i$. Then for all $m$ there exists $\psi_i \colon M_i \rightarrow M_{i+1}$ in $\radA^{\omega \alpha}$ for $i= 1,2,\dots,m$ with $\psi = \psi_m \dots \psi_1$. For $i=0,1,\dots,m$ consider the finitely presented subfunctor $t_i$ of $\mathbbm{1}_A$ corresponding to $f_i = \overline{\psi_i\dots\psi_1 \varphi}$. Then $f_i = \psi_i^n f_{i-1}$ and thus $t' = t_m \leq t_{m-1} \leq \dots \leq t_0 = t$. By induction, the interval $[t_i, t_{i+1}]$ has m-dimension greater than or equal to $\alpha$. Because $m$ was arbitrary, it follows that the interval $[t',t]$ has m-dimension greater than or equal to $\alpha + 1$.

Assume the claim holds for all $\alpha < \lambda$ with $\lambda$ a limit ordinal and let $\psi \in \radA^\lambda$. Then $\psi \in \radA^\alpha$ for all $\alpha < \lambda$ and by induction the m-dimension of $[t,t']$ is greater than or equal to $\alpha$ for all $\alpha < \lambda$. Hence, it is also greater than or equal to $\lambda$.
\end{proof}

\begin{coro}\label{kgtd} If Conjecture \ref{schr} (i) is true, then either $\KG(A) = \TD(A)$ or $\KG(A) = \TD(A)+1$.
\end{coro}

\begin{proof} If $\KG(A) = \infty$, then Conjecture \ref{schr} (i) implies $\radA^{\omega \alpha} \neq 0$ for all ordinal numbers $\alpha$. Hence, $\TD(A) \geq \alpha$ for all $\alpha$ by Proposition \ref{okay}. We conclude that $\TD(A) = \infty$. If $\KG(A) = \alpha+1$ for some ordinal number $\alpha$, then $\radA^{\omega \alpha} \neq 0$ by Conjecture \ref{schr} (i). Hence, $\TD(A) \geq \alpha$. Further, Proposition \ref{small} implies the inequality $\TD(A) \leq \alpha +1$. Thus, either $\TD(A) = \alpha$ or $\TD(A) = \alpha +1$.
\end{proof}

The following is a similar result for the torsion dimension as the result on the Krull-Gabriel dimension, that $\KG(A) \neq 1$ for all Artin algebras $A$. However, we need an additional assumption to conclude that $\TD(A) \neq 1$.

\begin{prop}\label{notone} If there exists $M\in \modA$ such that the smallest torsion class $\mathcal{T}$ containing $M$ is not functorially finite, then $\TD(A) > 1$.
\end{prop}
\begin{proof} Let $\mathcal{C} = \gen M$. By Proposition \ref{fff} it follows that $\mathcal{C}$ is a functorially finite epi-closed class. Hence, by Lemma \ref{prooop} and Lemma \ref{iff}, the ideal $\langle \mathcal{C} \rangle$ is a functorially finite torsion ideal. Consider the ideal
\begin{align*}
    \mathcal{I}_n = \langle \mathcal{C}\rangle\underbrace{\diamond \dots \diamond}_{n\text{\rm-times}} \langle \mathcal{C} \rangle
\end{align*}
for $n\in \mathbbm{N}$. Then $\mathcal{I}_n$ is a functorially finite torsion ideal by Corollary \ref{ext}. Further, $\mathcal{I}_n = \langle \mathcal{C}_n \rangle$ for
\begin{align*}
    \mathcal{C}_n = \mathcal{C} \underbrace{\diamond \dots \diamond}_{n\text{\rm-times}} \mathcal{C}
\end{align*}
by Lemma \ref{objext}. The functorially finite torsion ideal $\mathcal{I}_n$ corresponds to a finitely presented subfunctor $t_n$ of $\mathbbm{1}_A$ by Corollary \ref{corres}. Since $\mathcal{I}_{i} \subseteq \mathcal{I}_{i+1}$, this yields a chain $t_1 \leq t_2 \leq \dots$ of finitely presented subfunctors of $\mathbbm{1}_n$. We show that $t_n X \neq t_{n+1} X$ for infinitely many isomorphism types of indecomposable modules $X\in \modA$. 
Suppose the contrary holds. Then there exists $N>n$ such that $t_n X = t_{n+1} X$ for all indecomposable modules $X$ in $\mathcal{C}_{N+1}\backslash \mathcal{C}_N$. It follows that $t_{n+1} X = t_n X \in \mathcal{C}_n$ by Proposition \ref{objfun}. Further, since $X\in \mathcal{C}_{N+1}$, there exists a submodule $X'\in \mathcal{C}_{n+1}$ of $X$ such that $X/X' \in \mathcal{C}_{N-n}$. The canonical inclusion $X' \rightarrow X$ factors through the left $\langle \mathcal{C}_{n+1} \rangle$-approximation $t_{n+1} X \rightarrow X$, so $X' \subseteq t_{n+1} X$. Because $C_{N-n}$ it epi-closed, it follows that $X/t_{n+1} X \in \mathcal{C}_{N-n}$. Now $t_{n+1} X \in \mathcal{C}_n$ and $X/t_{n+1} X \in \mathcal{C}_{N-n}$ imply $X\in \mathcal{C}_{N}$. Hence, $\mathcal{C}_N = \mathcal{C}_{N+1}$. Thus,   $\mathcal{C}_N$ is closed under extensions. By Remark \ref{remrem} it follows that $\mathcal{T} = \mathcal{C}_N$, so $\mathcal{T}$ is functorially finite, which is a contradiction. Hence, the functors $t_n$ and $t_{n+1}$ always disagree on infinitely many isomorphism types of indecomposable modules. Next, we refine the chain of finitely presented functors
\begin{align*}
    t_1 < t_2 < \dots < \mathbbm{1}_A.
\end{align*}
Let $U\colon \modA \rightarrow \Ab$ be the forgetful functor. Then $U t_n$ is finitely presented, so there exists a maximal subfunctor $r_{n,1}$ of $U t_n$ containing $U t_{n-1}$. Now the quotient $U t_n/r_{n,1}$ is simple. Thus, the functors $U t_n$ and $r_{n,1}$ agree on all but one isomorphism type of an indecomposable module $Y$ \cite{Auslander3}. Notice that $r_{n,1} Y$ may not be a submodule of $Y$. To fix this, we define $t_{n,1}$ by $t_{n,1} X = r_{n,1} X = t_n X$ for $X\not\cong Y$ indecomposable and 
\begin{align*}
    t_{n,1} Y = t_{n-1} Y + \sum \varphi(t_{n,1} X),
\end{align*}
where we sum over all morphisms $\varphi\colon X\rightarrow Y$ with $X\not\cong Y$ indecomposable. By construction $t_{n-1} \subseteq t_{n,1} \subseteq r_{n,1}$ and $t_{n,1}$ is functorial. Further, $t_{n,1} X$ is a submodule of $X$ for all indecomposable modules $X$ and $t_{n,1}, t_n$ agree on all but one isomorphism type of an indecomposable module. It follows that $U t_n/U t_{n,1}$ is finitely presented and thus $U t_{n,1}$ is finitely presented. Further, the inequalities $t_{n-1} < t_{n,1} < t_n$ hold. Since $t_{n-1}$ and $t_n$ disagree on infinitely many different isomorphism types of indecomposable modules, we can proceed inductively to obtain a chain of finitely presented functors  
\begin{align*}
    t_{n-1} < \dots < t_{n,2} < t_{n,1} < t_n.
\end{align*}
It follows that the m-dimension of the interval $[t_{n-1},t_n]$ equals at least $1$ for all $n$. Thus, the m-dimension of the interval $[t_1, \mathbbm{1}_A]$ equals at least $2$. We conclude that  $\TD(A) > 1$.
\end{proof}

We are now ready to show that the torsion dimension coincides with the Krull-Gabriel dimension of a hereditary Artin algebra. For the details of the representation theory of hereditary Artin algebras, see \cite{Dlab2} and \cite{Ringel2}.

\begin{thm}\label{here} Let $A$ be a hereditary Artin algebra. Then $\TD(A) = \KG(A)$.
\end{thm}

\begin{proof}  We have $\TD(A) \leq \KG(A) $ by Proposition \ref{small}. Thus, if $A$ is of finite representation type, then $\TD(A) = \KG(A) = 0$.

If $A$ is of tame representation type, then $\TD(A) \leq \KG(A) = 2$. Let $M$ be the direct sum of all simple regular modules (up to isomorphism) of a tube of finite rank. Then the smallest torsion class containing $M$ is not functorially finite by \cite[Proposition 4]{Ringel3}. Hence, $\TD(A) > 1$ by Proposition \ref{notone} and thus $\TD(A) = 2.$

If $A$ is of wild representation type, then $\radA^\infty \neq 0$ by \cite[Proposition 8.15]{Krause}. Hence, $\TD(A) = \infty$ by Proposition \ref{okay} and thus $\TD(A) = \infty  = \KG(A)$.
\end{proof}

In the following example, we show an application of the monoidal structure on pairs $s\leq t$ of subfunctor of $\mathbbm{1}_A$ (see Section 2) to the torsion dimension and the Krull Gabriel dimension.

\begin{exa}\label{end}\rm  Let $k$ be an algebraically closed field, $Q = \begin{tikzcd}
    1 \arrow[r, shift left] \arrow[r, shift right] & 2
\end{tikzcd}$ the Kronecker quiver and $A = kQ$ (as in Example \ref{kro}). Then $A$ is a tame hereditary Artin algebra. Hence, $\TD(A) = \KG(A) = 2$ by Theorem \ref{here}. In this example, we show that two finitely presented subfunctors of $\mathbbm{1}_A$ are sufficient to produce enough finitely presented subfunctors of $\mathbbm{1}_A$ by the monoidal structure to deduce $\TD(A) \geq 2$.  

Let $S \in \modA$ be simple injective and $T\in \modA$ simple projective. We can divide the indecomposable modules of $\modA$ into three parts: The preprojective modules $\mathcal{P}$, the regular modules  $\mathcal{R}$ and the preinjective modules $\mathcal{I}$. Further, the regular modules $\mathcal{R}$ can be divided into tubes $\mathcal{R}^\lambda = \{R^\lambda_1, R^\lambda_2, \dots\}$ for $\lambda\in k\cup \{\infty \}$. For all $\lambda$ the full subcategory of modules isomorphic to a direct sum of modules in $\mathcal{R}^\lambda$ is abelian and has one simple object $R_1^\lambda$. Further, it is uniserial, so the lattice of subobjects of $R_j^\lambda$ is linearly ordered for all $j$. The linear order is determined by the short exact sequences 
\begin{align*}
    0 \longrightarrow R^\lambda_i \longrightarrow R^\lambda_j \longrightarrow R^\lambda_{j-i} \longrightarrow 0.
\end{align*}
Consider $\mathcal{C} = \gen R_1^\lambda$ and $\mathcal{D} = \cogen R_1^\lambda$. The indecomposable modules in $\mathcal{C}$ equal $\{R_1^\lambda, S\}$ and the ones in $\mathcal{D}$ equal $\{R_1^\lambda, T\}$ (up to isomorphism). Further, $\mathcal{C}$ is a functorially finite epi-closed class and $\mathcal{D}$ a functorially finite mono-closed class by Proposition \ref{fff}. Let $\mathcal{I} = \langle \mathcal{C} \rangle$ and $\mathcal{J} = \langle \mathcal{C} \rangle$. Then, by Lemma \ref{prooop} and Lemma \ref{iff}, the ideal $\mathcal{I}$ is a functorially finite torsion ideal and $\mathcal{J}$ a functorially finite torsion-free ideal. Now $\mathcal{I}$ corresponds to a finitely presented subfunctor $t$ of $\mathbbm{1}_A$ and $\mathcal{J}$ to a finitely presented subfunctor $r$ of $\mathbbm{1}_A$ by Corollary \ref{corres}. For $M\in \modA$ the module $t M$ is the largest submodule of $M$ in $\mathcal{C}$ and $M/r M$ is the largest factor module of $M$ in $\mathcal{D}$ by Proposition \ref{objfun}. By the above short exact sequences, we conclude that
\begin{align*}
    t R_j^\lambda = R_1^\lambda, \qquad r R_j^\lambda = R_{j-1}^\lambda.
\end{align*}
Next, we employ the monoidal structure on pairs of subfunctors of $\mathbbm{1}_A$. Let $t_i$ be defined by $(\mathbbm{1}_A/t)^i = \mathbbm{1}_A/t_i$. Then $t_1 \leq t_2 \leq \dots$ are finitely presented by Lemma \ref{rest}. By the values of $t$, it follows that $t_i R_k^\lambda = R_k^\lambda$ if $k < i$ and otherwise $t_i R_k^\lambda = R_i^\lambda$. Further, let $r_{i,j}$ be defined by $r^j (t_{i+j+1}/t_i) = r_{i,j}/t_i$. Then $r_{i,j}$ is finitely presented by Lemma \ref{rest}. By the values of $t_i$ and $r$, it follows that 
\begin{alignat*}{2}
    k < i  \colon \quad &r_{i,j} R_k^\lambda = R_k^\lambda &&= t_i R_k^\lambda,\\
    i \leq k \leq i+j \colon \quad &r_{i,j} R_k^\lambda = R_{j}^\lambda &&= t_i R_k^\lambda, \\
    i+j < k \colon \quad &r_{i,j} R_k^\lambda = R_{i+1}^\lambda &&= t_{i+1} R_k^\lambda.
\end{alignat*}
Hence, if we restrict $r_{i,j}$ on $\mathcal{R}^\lambda$, then
\begin{align*}
    r_{i,0} \mid_{\mathcal{R}^\lambda}\, \gneq r_{i,1} \mid_{\mathcal{R}^\lambda}\, \gneq r_{i,2} \mid_{\mathcal{R}^\lambda}\, \gneq r_{i,3} \mid_{\mathcal{R}^\lambda}\, \gneq \dots 
\end{align*}
and thus
\begin{align*}
    t_{i+1} = r_{i,0} \gneq r_{i,0} \cap r_{i,1} \gneq r_{i,0} \cap r_{i,1} \cap r_{i,2} \gneq \dots \gneq t_i.
\end{align*}
It follows that the m-dimension of the interval $[t_i, t_{i+1}]$ equals $1$ for all $i$. We conclude that the m-dimension of $[t, \mathbbm{1}_A]$ equals $2$ and so $\TD (A) \geq 2$.
\end{exa}

\end{document}